%% file: NewtonAPG.tex
\PassOptionsToPackage{table}{xcolor}
\documentclass[
onefignum,
onetabnum,
nohypdvips]{siamart171218}

\input{ex_shared}

\usepackage{mathtools}

\begin{document}
	
	\maketitle
	
\begin{abstract}
	Finding the stationary states of a free energy functional is an important
	problem in phase field crystal (PFC) models. Many efforts have been devoted
	for designing numerical schemes with energy dissipation and mass
	conservation properties. However, most existing approaches are
	time-consuming due to the requirement of small effective step sizes. In this
	paper, we discretize the energy functional and propose efficient numerical
	algorithms for solving the constrained non-convex minimization problem.
	A class of gradient based approaches, which is the so-called adaptive accelerated Bregman
	proximal gradient (AA-BPG) methods, is proposed and the convergence property
	is established without the global Lipschitz constant requirements.
	A practical Newton method is also designed to further accelerate the local convergence with
	convergence guarantee.
	One key feature of our algorithms is that the energy dissipation and mass conservation
	properties hold during the iteration process.
	{Moreover, we develop a hybrid acceleration framework to accelerate the AA-BPG methods and most
	of existing approaches through coupling with the practical Newton method.} 
	Extensive numerical experiments, including two three dimensional periodic crystals in
	Landau-Brazovskii (LB) model and a two dimensional quasicrystal in Lifshitz-Petrich (LP) model,
	demonstrate that our approaches have adaptive step sizes which lead to a significant
	acceleration over many existing methods when computing complex structures.
\end{abstract}

\begin{keywords}
	Phase field crystal models, Stationary states, 
    {Adaptive accelerated Bregman proximal gradient methods, Preconditioned conjugate gradient
	method, Hybrid acceleration framework}.
\end{keywords}
 

\section{Introduction}
\label{sec:intr}

The phase field crystal (PFC) model is an important approach to
describe many physical processes and material properties, 
such as the formation of ordered structures,
nucleation process, crystal growth, elastic and plastic
deformations of the lattice, dislocations \cite{chen2002phase, provatas2010phase}. 
More concretely, let the order parameter function be
$\phi(\br)$, the PFC model can be expressed by a free energy functional
\begin{align} 
E(\phi; \Theta) = G(\phi; \Theta) + F(\phi; \Theta),
\label{eq:energy}
\end{align}
where $\Theta$ are the physical parameters, $F[\phi]$ is the bulk
energy with polynomial type or log-type formulation and $G[\phi]$ is the
interaction energy that contains higher-order differential operators to form ordered
structures \cite{brazovskii1975phase, lifshitz1997theoretical,
	swift1977hydrodynamic}.
A typical interaction potential function for a domain $\Omega$ is 
\begin{equation}
G(\phi) = \frac{1}{|\Omega|}\int_\Omega
\Big[\prod_{j=1}^{m}(\Delta+q_j^2)\phi \Big]^2 \,d\br, 
~~ m\in\mathbb{N}
\end{equation}
which can be used to describe the pattern formation of periodic crystals,
quasicrystals and multi-polynary crystals\,\cite{lifshitz1997theoretical,
mkhonta2013exploring}. In order to understand the theory of PFC models as well as 
predict and guide experiments, it requires to find stationary states $\phi_s(\br; \Theta)$ and construct phase diagrams
of the energy functional \cref{eq:energy}.
Denote $V$ to be a feasible space,
the phase diagram is obtained via solving the minimization problem
\begin{align}
\min_{\phi} E(\phi; \Theta), \mbox{ s.t. } \phi\in V,
\label{eq:minEnergy}
\end{align}
with different physical parameters $\Theta$, which brings
the tremendous computational burden.
Therefore, within an appropriate spatial discretization, the goal of this paper is to
develop efficient and robust numerical methods for solving \cref{eq:minEnergy} with
guaranteed convergence while keeping the desired dissipation and conservation
properties during the iterative process.

Most existing numerical methods for computing the stationary states
of PFC models can be classified into two categories. 
One is to solve the steady
nonlinear Euler-Lagrange equations of \cref{eq:minEnergy}
through different spatial discretization approaches.
The other class aims at solving the nonlinear gradient flow equation by using the numerical PDE methods. 
In these approaches, there have been extensive works of energy stable numerical schemes 
for the time-dependent PFC model and{ its various extensions, such as the modified PFC 
(MPFC)\,\cite{wang2011energy,lee2017first,guo2018high} and square PFC (SPFC) models\,\cite{cheng2019energy}.}
Typical energy stable schemes to gradient flows include convex
splitting methods\,\cite{wise2009energy,shin2016first},  and stabilized factor methods in both 
the first and second order temporal accuracy orders\,\cite{shen2010numerical}, {the exponential time 
differencing  schemes\,\cite{du2004stability}, }and recently developed
invariant energy quadrature\,\cite{yang2016linear} and scalar auxiliary variable approaches\,\cite{shen2019new} for a modified energy. 
It is noted that the gradient flow approach is able to describe the
quasi-equilibrium behavior of PFC systems.  
Numerically, the gradient flow is discretized in both
space and time domain via different discretization techniques and the
stationary state is obtained with a proper choice of initialization. 
{Many popular spatial approximations have been used, such as the finite difference 
method\,\cite{wise2009energy,wang2011energy,hu2009stable}, the finite element 
method\,\cite{feng2007analysis,du2008adaptive} and Fourier 
pseudo-spectral  method\,\cite{chen1998applications,jiang2014numerical,cheng2019energy}.}

Under an appropriate spatial
discretization scheme, the infinite dimensional problem
\cref{eq:minEnergy} can be formulated as a 
minimization problem in a finite dimensional space.
Thus, there may exist alternative numerical methods that can
converge to the steady states quickly
by using modern optimization techniques. For example, similar ideas have
shown success in computing steady states of 
the Bose-Einstein condensate \cite{wu2017regularized} and the calculation of
density functional theory \cite{ulbrich2015proximal,liu2015analysis}. In this
paper, in order to keep the mass conservation property, an additional constraint
is imposed in \cref{eq:minEnergy} and the detail will be given in the next
section. Inspired by the recent advances of gradient based methods which have been
successfully applied in image processing and machine learning, we propose an adaptive
accelerated Bregman proximal gradient (AA-BPG) method for
computing the stationary states of \cref{eq:minEnergy}. In each iteration, the
AA-BPG method updates the estimation of the order parameter function by solving
linear equations which have closed form when using the pseudo-spectral
discretization and chooses step sizes by using the line search algorithm
initialized with the Barzilai-Borwein (BB) method \cite{barzilai1988two}.  Meanwhile, a
restart scheme is proposed such that the iterations satisfy energy dissipation
property and it is proved that the generated sequence converges to a stationary
point of \cref{eq:minEnergy} without the assumption of the existence of global
Lipschitz constant of the bulk energy $F$. Moreover, an  regularized Newton
method is applied for further accelerating the local convergence. More specifically, an 
preconditioned conjugate gradient method is designed for solving the regularized Newton 
system efficiently. Extensive numerical experiments have demonstrated that our approach can 
quickly reach the vicinity of an optimal solution with moderately accuracy, even for very challenge cases. 

The rest of this paper is organized as follows. 
In \cref{sec:discretization}, we present the PFC models considered in
this paper, and the projection method discretization. 
In \cref{sec:first-order}, we present the AA-BPG method for solving the
constrained non-convex optimization
with proved convergence. 
In \cref{sec:applicationPFC}, two choices of Bregman distance are proposed and applied for the PFC problems.
In \cref{sec:second-order}, we design a practice Newton preconditioned conjugate gradient 
(Newton-PCG) method with gradient convergence guarantee.  Then,
a hybrid acceleration framework is proposed to
further accelerate the calculation. Numerical results are reported in
\cref{sec:result} to illustrate the efficiency and accuracy of our
algorithms. Finally, some concluding remarks are given in \cref{sec:conclusion}.

\subsection{Notations and definitions}
{
Let $ C^k $ be  the set of  $ k $-th continuously differentiable functions on the whole space.
The domain of a real-valued  function $ f$ is defined as $ \dom
f:=\{x:f(x)<+\infty\} $. We say $ f $ is proper if $ f > -\infty $ and $  \dom f\neq
\emptyset$.  For $ \alpha\in\bbR $, let $  [f \leq \alpha]:= \{x : f(x) \leq \alpha\}$
be the $ \alpha $-(sub)level set of $ f $. We say that $ f $ is level  bounded
if $ [f\leq \alpha] $ is bounded for all $ \alpha\in \bbR$.  $ f $ is lower
semicontinuous if all level sets of $ f $ are closed. For a proper function $ f $,  the subgradient\,\cite{brezis2010functional} of $ f $
at $ x \in \dom f $ is defined as $ \partial f(x) = \{u : f(y) - f(x) -
\langle u, y- x\rangle \geq 0, \forall\, y \in \dom f\} $. A point $x$ is called a stationary point of $f$ if 
$0\in\partial f(x)$.
}
\section{Problem formulation}
\label{sec:discretization}

\subsection{Physical models} 
Two classes of PFC models are considered in the paper. The first
one is the Landau-Brazovskii (LB) model which can characterize 
the phase and phase transitions of periodic
crystals\,\cite{brazovskii1975phase}.
It has been discovered
in many different scientific fields, e.g., polymeric 
materials\,\cite{shi1996theory}. 
In particular, the energy functional of LB model is 
\begin{align}
E_{LB}(\phi) = \frac{1}{|\Omega|}\int_\Omega
\left\{\underbrace{\frac{\xi^2}{2}[(\Delta +1)\phi]^2}_{G(\phi)} + \underbrace{\frac{\tau}{2!}\phi^2 
-\frac{\gamma}{3!}\phi^3 + \frac{1}{4!}\phi^4}_{F(\phi)} \right\}\,
d \br,
\label{eq:LB}
\end{align}
where $\phi(\br)$ is a real-valued function which measures the
order of system in terms of order parameter.
$\Omega$ is the bounded domain of system, $\xi$ is the bare
correlation length, $\tau$ is the dimensionless reduced
temperature, $\gamma$ is phenomenological coefficient.  
Compared with double-well bulk energy\,\cite{swift1977hydrodynamic},
the cubic term in the LB functional helps us study
the first-order phase transition. 

The second one is the Lifshitz-Petrich (LP) model that can
simulate quasiperiodic structures, such as the bi-frequency 
excited Faraday wave\,\cite{lifshitz1997theoretical}, and 
explain the stability of soft-matter
quasicrystals\,\cite{lifshitz2007soft, jiang2015stability}. 
Since quasiperiodic structures are space-filling without decay, it is necessary to
define the average spacial integral over the whole space as 
$ \bbint = \lim_{R\rightarrow \infty}\frac{1}{|B_R|}\int_{B_R}, $
where $B_R\subset\bbR^d$ is the ball centred at origin with radii $R$.
The energy functional of LP model is given by 
\begin{align}
E_{LP}(\phi) = \bbint
\left\{\underbrace{\frac{c}{2}[(\Delta +q_1^2)(\Delta + q_2^2)\phi]^2}_{G(\phi)} +\underbrace{
\frac{\varepsilon}{2}\phi^2 
-\frac{\kappa}{3}\phi^3 + \frac{1}{4}\phi^4}_{F(\phi)} \right\}\,
d \br,
\label{eq:LP}
\end{align}
where $c$ is the energy penalty, $\varepsilon$ and $\kappa$ are
phenomenological coefficients. 

Furthermore, we impose the following mean zero condition of order parameter on
LB and LP systems to ensure the mass conservation, respectively.
\begin{equation}
\label{mass-conservation}
	\dfrac{1}{|\Omega|}\int_{\Omega} \phi(\br) d\br = 0 \quad\text{ or }\quad  \bbint \phi(\br) d\br=0.
\end{equation}
The equality constraint condition is from the definition of the order parameter which
is the deviation from average density.

\subsection{Projection method discretization}
\label{subsec:PM}
In this section, we introduce the projection method \cite{jiang2014numerical}, a high
dimensional interpretation approach which can avoid the Diophantine approximation
error in computing quasiperiodic systems, to discretize the LB and LP energy
functional. It is noted that the stationary states in LB model is periodic, and thus
it can be discretized by the Fourier pseudo-spectral method which is a  special case
of projection method. 
Therefore, we only consider the projection method discretization of the LP model
\cref{eq:LP}.
We immediately have the following orthonormal property in the average spacial
integral sense
\begin{align}
\bbint e^{i\bk \cdot \br} e^{-i\bk' \cdot \br}\,d\br =
\delta_{\bk \bk'}, ~~~ \forall \bk, \bk' \in \bbR^d.
\label{eq:AP:orth}
\end{align}
For a quasiperiodic function, we can define the Bohr-Fourier transformation as \cite{katznelson2004anintroduction}
\begin{align}
\hphi(\bk) = \bbint \phi(\br)e^{-i\bk\cdot \br}\,d\br,~~
\bk\in\bbR^d.
\end{align}
In this paper, we carry out the above computation in a higher
dimension using the projection method which
is based on the fact that a $d$-dimensional
quasicrystal can be embedded into an $n$-dimensional periodic
structure ($n \geqslant d$)\,\cite{hiller1985crystallographic}. 
The dimension $n$ is the number of linearly independent numbers over the rational number field.
Using the projection method, the order parameter $\phi(\br)$ can be expressed as
\begin{equation}
\phi(\br) = \sum_{\bh\in\bbZ^n} \hphi(\bh)
e^{i[(\mathcal{P}\cdot\mathbf{B}\bh)^\top\cdot\br]},
~~\br\in\mathbb{R}^d,
\label{eq:pm}
\end{equation}
where $\mathbf{B}\in\bbR^{n\times n}$ is invertible, related to the $n$-dimensional
primitive reciprocal lattice. The corresponding computational domain in physical
space is $2\pi\bB^{-T}\tau$, $\tau\in[0,1)^n$. 
The projection matrix $\mathcal{P}\in\bbR^{d\times n}$ depends
on the property of quasicrystals, such as rotational symmetry \cite{hiller1985crystallographic}. 
If consider periodic crystals, the projection matrix becomes the $d$-order identity
matrix, then the projection reduces to the common Fourier spectral method.
The Fourier coefficient $\hphi(\bh)$ satisfies
\begin{align}
X := \left\{(\hphi(\bh))_{\bh\in\bbZ^n}:
\hphi(\bh)\in\mathbb{C}, ~
\sum_{\bh\in\bbZ^n}|\hphi(\bh)|<\infty \right\}.
\end{align}
In practice,  let
$\bN=(N_1, N_2, \dots, N_n)\in \bbN^n$, and 
\begin{align} 
X_{\bN} := \{\hphi(\bh)\in X:  \hphi(\bh) = 0, 
~\mbox{for
	all}~ |h_j|> N_j/2, ~ j=1,2,\dots,n \}.  
\end{align}
The number of elements in the set is $N=(N_1+1)(N_2+1) \cdots (N_n+1)$. 
Together with \cref{eq:AP:orth} and \cref{eq:pm}, the discretized energy function
\cref{eq:LP} is 
\begin{equation}
	E_{\bh}(\hPhi) = G_{\bh}(\hPhi) + F_{\bh}(\hPhi),
\end{equation}
where $G_h$ and $F_h$ are the discretized interaction
and bulk energies:
\begin{equation}
\begin{aligned}
	G_{\bh}(\hPhi) & = \frac{c}{2}\sum_{\bh_1 + \bh_2 = 0}
\left[q_1^2-(\mathcal{P}\mathbf{B}\bh)^\top (\mathcal{P}\mathbf{B}\bh)\right]^2
\left[q_2^2-(\mathcal{P}\mathbf{B}\bh)^\top (\mathcal{P}\mathbf{B}\bh)\right]^2
\hphi(\bh_1)\hphi(\bh_2) \\
F_{\bh}(\hPhi) & = \frac{\varepsilon}{2}\sum_{\bh_1+\bh_2={\bm 0}}\hphi(\bh_1)\hphi(\bh_2) 
-\frac{\kappa}{3}\sum_{\bh_1+\bh_2+\bh_3={\bm 0}}\hphi(\bh_1)\hphi(\bh_2)\hphi(\bh_3) 
\\
&
+\frac{1}{4}\sum_{\bh_1+\bh_2+\bh_3+\bh_4={\bm
		0}}\hphi(\bh_1)\hphi(\bh_2)\hphi(\bh_3)\hphi(\bh_4),
\end{aligned}
\label{eq:LPfinite}
\end{equation}
and $\bh_j\in\bbZ^n$, , $\hphi_j\in X_{\bN}$, $j=1,2,\dots,4$,
$\hat{\Phi}=(\hphi_1, \hphi_2, \dots, \hphi_N)\in\bbC^{N}$. 
It is clear that the nonlinear terms in $F_h$ are
$n$-dimensional convolutions in the reciprocal space. A direct
evaluation of these convolution terms is extremely expensive.
Instead, these terms are simple multiplication in the
$n$-dimensional physical space. Similar to the pseudo-spectral approach, these
convolutions can be efficiently calculated through FFT. Moreover,
the mass conservation constraint \cref{mass-conservation} is discretized as 
\begin{equation}
e_1^\top \hPhi = 0,
\end{equation} 
where $e_1=(1,0,\ldots,0)^\top \in\bbR^N$. Therefore, we obtain  the following finite
dimensional minimization problem
\begin{equation}\label{min:finite}
\min_{\hPhi\in \mathbb{C}^{N}} 
E_{\bh}(\hPhi) = G_{\bh}(\hPhi) + F_{\bh}(\hPhi), \mbox{ s.t. } e_1^\top \hPhi = 0.
\end{equation}
For simplicity, we omit the subscription in $G_{\bh}$ and $F_{\bh}$ in the following context.
According to \cref{eq:LPfinite}, denoting $ \calF_N\in\mathbb{C}^{ N\times N} $ as the
discretized Fourier transformation matrix, we have
\begin{align}
 \nabla G(\hPhi) = D\hPhi, & \quad \nabla F(\hPhi) = \calF_N^{-1}\Lambda\calF_N\hPhi\label{eq:Grad} \\
\nabla^2 G(\hPhi) = D, &\quad \nabla^2 F(\hPhi) = \calF_N^{-1}\Lambda^{(')}\calF_N, \label{eq:Hessian}
\end{align}
where $ D $ is a diagonal matrix with nonnegative entries $
c\left[q_1^2-(\mathcal{P}\mathbf{B}\bh)^\top (\mathcal{P}\mathbf{B}\bh)\right]^2\times \\
\left[q_2^2-(\mathcal{P}\mathbf{B}\bh)^\top (\mathcal{P}\mathbf{B}\bh)\right]^2 $ and 
$ \Lambda,\Lambda^{(')}\in\mathbb{R}^{ N\times N} $ are also  diagonal matrices but related to $ \hPhi $.
In the next section, we propose the adaptive accelerated Bregman proximal gradient (AA-BPG) 
method for solving the constrained minimization problem \cref{min:finite}.

\section{The AA-BPG method}\label{sec:first-order}

Consider the minimization problem that has the form
\begin{equation}\label{GeneralFormulation}
\min_x E(x) = f(x) + g(x),
\end{equation}
where $f\in C^2$ is proper but non-convex and $g$ is proper, lower
semi-continuous and convex. Let the domain of $E$ to be $\dom E
=\{x~|~E(x)<+\infty\}$, we make the following assumptions.
\begin{assumption}\label{assum1}
	$ E $ is bounded below and for any $ x^0 \in\dom E$, the sub-level set $ \calM(x^0): = \{x| E(x)\leq E(x^0)\} $ is compact.
\end{assumption}
Let $h$ be a strongly convex function such that $\dom h\subset\dom f$ and $\dom
g\,\cap \,\intdom h\neq\emptyset$.  Then, it induces the \emph{Bregman divergence}\,\cite{bregman1967relaxation} defined as 
	\begin{align}\label{eqn:bregmandiv}
	D_h(x, y) = h(x)-h(y) - \langle \nabla h(y), x-y\rangle,~\forall (x,y)\in\dom h \times \intdom h.
	\end{align}
	It is noted that $D_h(x,y)\geq0$ and $D_h(x,y)=0$ if and only if $x=y$ due to the
	strongly convexity of $h$. Furthermore, $ D_h(x,\bar{x}) \to 0$ as $ x\to \bar{x}
	$. In recent years, Bregman distance based proximal methods \cite{Heinz,bolte2018first} have 
	been proposed and applied  for solving the \cref{GeneralFormulation} in a general non-convex 
	setting \cite{Bregman}. Basically, given the current estimation $x^k\in\intdom h$ and step 
	size $\alpha_k>0$, it updates $x^{k+1}$ via 
	\begin{equation}\label{Iter:Breg}
	x^{k+1} = \argmin_x \left\{g(x)+\langle x-x^k,\nabla f(x^k)\rangle + \dfrac{1}{\alpha_k} D_h(x,x^k)\right\}.
	\end{equation}
	Under suitable assumptions, it is proved in \cite{Bregman} that the iterates
	$\{x^k\}$ has similar convergence property as the traditional proximal gradient
	method \cite{beck2009fast} while iteration \cref{Iter:Breg} does not require the
	Lipschitz condition on $\nabla f$. Motivated by the Nesterov acceleration
	technique \cite{tseng2008accelerated,beck2009fast}, we add an extrapolation step
	before \cref{Iter:Breg} and thus the iterate becomes
	\begin{equation}\label{Iter:ABreg}
	\begin{aligned}
	y^k & = x^k + w_k(x^k-x^{k-1}), \\
	x^{k+1} &= \argmin_x \left\{g(x)+\langle x-y^k,\nabla f(y^k)\rangle + \dfrac{1}{\alpha_k} D_h(x,y^k)\right\},
	\end{aligned}
	\end{equation}
	where $w_k\in[0,\bar{w}]$. It is noted that the minimization problems in
	\cref{Iter:Breg} and \cref{Iter:ABreg} are well defined and single valued
	as $g$ is convex  and $h$ is strongly convex.
	Although the extrapolation step accelerates the convergence in some cases, it may generate the
	oscillating phenomenon of the objective value $E(x)$ that slows down the
	convergence~\cite{o2015adaptive}. Therefore, we propose a restart algorithm that
	leads to a convergent algorithm for solving \cref{GeneralFormulation} with energy dissipation property. Given $\alpha_k>0$, define
	\begin{equation}\label{BProxsubprob}
	z^{k} = \argmin_x \left\{g(x)+\langle x-y^k,\nabla f(y^k)\rangle + \dfrac{1}{\alpha_k} D_h(x,y^k)\right\},
	\end{equation}
	we reset $w_k=0$ if the following does not hold
	\begin{equation}\label{criterion:line_search}
	E(x^k)-E(z^k)\geq c\|x^k-x^{k+1}\|^2 
	\end{equation}
	for some constant $c>0$.  In the next section, we will show that
	\cref{criterion:line_search} holds when $w_k=0$. Overall, the AA-BPG
	algorithm is presented in \cref{alg:ABPG}.
	
	\noindent{\bf Step size estimation.}
		 In each step, $\alpha_k$ is chosen adaptively by backtracking linear search method  which is initialized by the BB step\,\cite{barzilai1988two} estimation, i.e.
	\begin{align}\label{BBSteps}
	\alpha_k = \dfrac{\langle s_k, s_k\rangle}{\langle s_k, v_k \rangle}\text{ or } \dfrac{\langle v_k, s_k\rangle}{\langle v_k, v_k\rangle},
	\end{align}
	where $ s_k = x^k - x^{k-1} $ and $ v_k = \nabla f(x^k) - \nabla f(x^{k-1}) $.
	Let $\eta>0$ be a small constant and $z^k$ be obtained from \cref{BProxsubprob}, we adopt the step size $\alpha_k$ whenever the following inequality holds
	\begin{equation}\label{condition:linesearch}
	E(y^k)-E(z^{k})\geq\eta\|y^k-z^{k}\|^2.
	\end{equation}
	The detailed estimation method is presented in \cref{alg:esalpha}.
	\begin{algorithm}[!pbht]
		\caption{AA-BPG Algorithm}
		\label{alg:ABPG}
		\begin{algorithmic}[1]
			\REQUIRE $x^1=x^0$, $ \alpha_0 >0 $, $w_0=0$, $\rho\in (0,1)$, $ \eta,c,\bar w>0$ and $k = 1$.
			\WHILE {the stop criterion is not satisfied}
			\STATE Update $y^k = x^k+w^k(x^k-x^{k-1})$
			\STATE Estimate $ \alpha_k $ by \cref{alg:esalpha}
			\STATE Calculate $z^k$ via \cref{BProxsubprob}
			\IF {\cref{criterion:line_search} holds}
			\STATE $ x^{k+1} = z^k $ and update $w_{k+1}\in [0,\bar w]$.
			\ELSE
			\STATE $ x^{k+1} = x^k $ and reset $w_{k+1}=0$.
			\ENDIF	
			\STATE $ k =k+1 $.
			\ENDWHILE	
		\end{algorithmic}
	\end{algorithm}

\begin{algorithm}[!pbht]
	\caption{Estimation of $\alpha_k$ at $y^k$}
	\label{alg:esalpha}
	\begin{algorithmic}[1]
		\REQUIRE $x^k$, $y^k$, $\eta>0$ and $\rho\in(0,1)$ and $\alpha_{\min},\alpha_{\max}>0$
		\STATE Initialize $ \alpha_{k} $ by BB step \cref{BBSteps}.
		\FOR{$j=1,2\ldots$}
		\STATE Calculate $z^k$ via \cref{BProxsubprob}
		\IF{\cref{condition:linesearch} holds or $\alpha_k<\alpha_{\min}$}
		\STATE {\bf break}
		\ELSE
		\STATE $\alpha_k = \rho \alpha_k$
		\ENDIF 
		\ENDFOR
		\STATE Output $\alpha_k=\max(\min(\alpha_k,\alpha_{\max}),\alpha_{\min})$.
	\end{algorithmic}
\end{algorithm}

\subsection{Convergence analysis}
In this section, we focus on the convergence analysis of the proposed AA-BPG method. Before proceeding, 
we introduce a significant definition used in analysis.

\begin{definition}\label{def:relsmooth}
A function $ f\in C^2 $ is $R_f$-relative smooth if there exists a strongly convex function $ h\in C^2 $ such that 
\begin{align}\label{eqn:relsmooth}
R_f \nabla^2 h(x) - \nabla^2 f(x) \succeq 0,\quad \forall x\in\intdom h.
\end{align}
\end{definition}
Throughout this section, we impose the next assumption on $ f $.
\begin{remark}
    {
	If $ h = \|\cdot\|^2/2 $, the relative smoothness becomes the Lipschitz
	smoothness.
	}
\end{remark}
\begin{assumption}\label{assum2}
	There exists $R_f>0$ such that $f$ is $ R_f $-relative smooth with respect to a strongly convex function $ h\in C^2 $.
\end{assumption}
\begin{remark}
    {
    In the LB model \cref{eq:LB} and LP model \cref{eq:LP},
	their bulk energies are fourth degree polynomials and  its gradient are not Lipschitz continuous.
	However, we will show that relative smoothness constant $R_f$ can be $O(1)$ through appropriately choosing the strongly convex function
	$h$ in \cref{PFCproperty2}.
	}
\end{remark}

\subsection{Convergence property}

In this subsection, {we will prove the convergence property of the \cref{alg:ABPG}. The outline of the proof is given in \cref{fig:ProofProcess}.}
\begin{figure}[!htbp]
	\centering
	\includegraphics[width=5.1in]{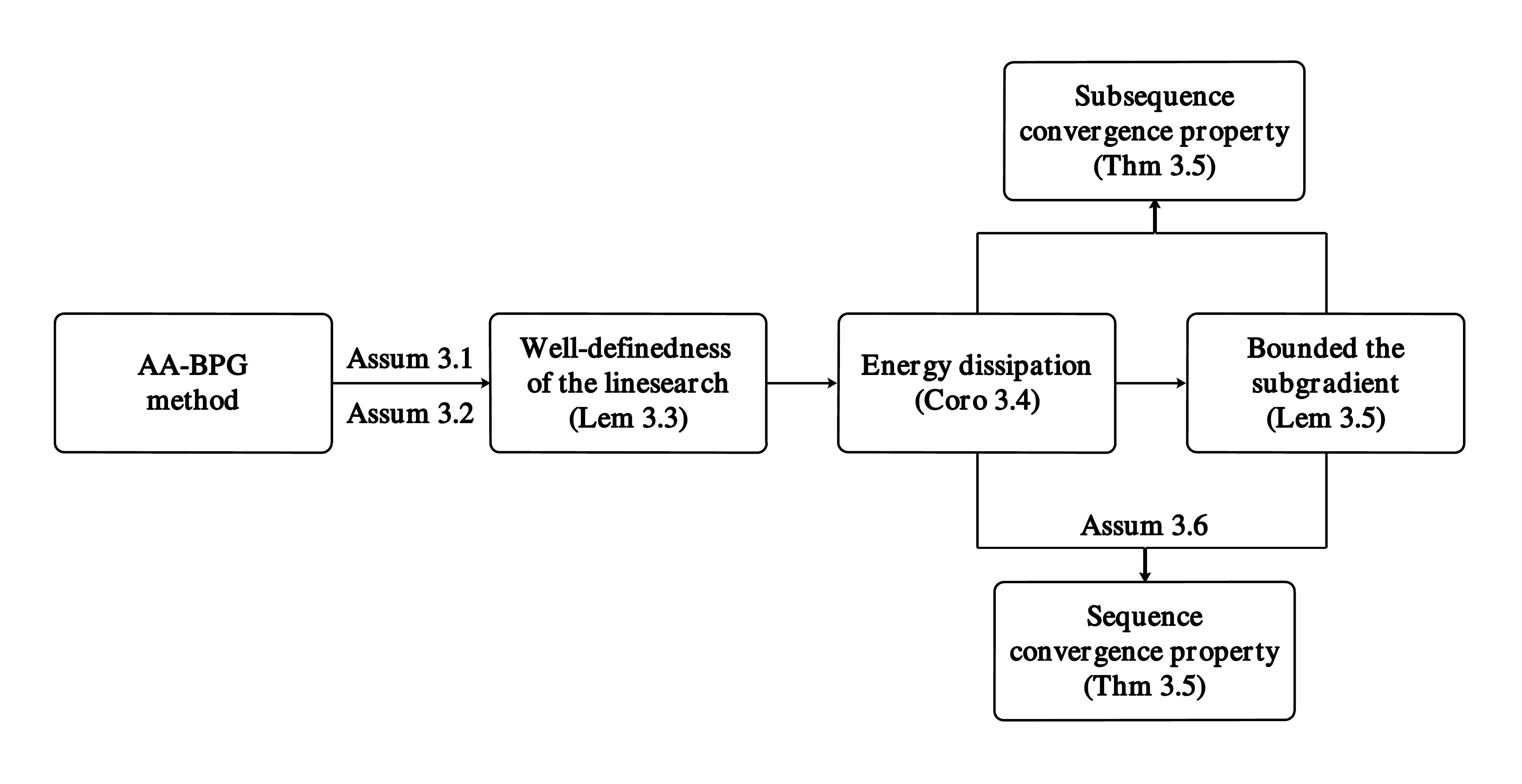}
	\caption{{The flow chart of the convergence proof of \cref{alg:ABPG}.}
	}
	\label{fig:ProofProcess}
\end{figure}
Under the \cref{assum2},  we have the following useful lemma as stated in \cite{Heinz}.
 \begin{lemma}[\cite{Heinz}]
	\label{lem:rel-sm-Breg}
	If $ f $ is $ R_f $-relative smooth with respect to $ h $, then 
	\begin{align}\label{relatively-bounded}
	f(x)-f(y) - \langle \nabla f(y), x-y\rangle \leq R_f D_h(x,y), \quad \forall x, y\in\intdom h.
	\end{align}
 \end{lemma}
Based on the above Lemma, the descent property of the iteration generated by Bregman
proximal operator \cref{BProxsubprob} is established as follows.
\begin{lemma}\label{lem:BPGsuffdesc}
Let $\alpha>0$ and suppose the  \cref{assum2} holds. If 
	\begin{equation}\label{BProx}
z = \argmin_x \left\{g(x)+\langle x-y,\nabla f(y)\rangle + \dfrac{1}{\alpha} D_h(x,y)\right\},
\end{equation}
then there  exists some $\sigma>0$ such that  
\begin{align}\label{BPGsuffDec}
E(y) - E(z)\geq  \left(\dfrac{1}{\alpha} - R_f\right)\frac{\sigma}{2}\|z-y\|^2.
\end{align}
\end{lemma}
\begin{proof}
Since $h$ is strongly convex, there exists some constant $\sigma>0$ such that $h(x)-\sigma\|x\|^2/2$ is convex. 
Then, $\nabla^2 h(x)-\sigma I\succeq 0$ and we have
\begin{equation}\label{strongconvex}
D_h(z,x) = h(z) - h(y) - \langle \nabla h(y), z-y\rangle\geq \frac{\sigma}{2} \|z-y\|^2.
\end{equation}
From  the optimal condition of \cref{BProx}, we have
\begin{align*}
E(y) &= f(y) + g(y) = \left[f(y) + \langle \nabla f(y), x - y\rangle + \dfrac{1}{\alpha}D_h(x,y) + g(x) \right]_{x = y}\\
& \geq f(y) + \langle \nabla f(y), z - y\rangle + \dfrac{1}{\alpha}D_h(z, y) + g(z)\\
&\geq f(z) - R_f D_h(z,y)+ \dfrac{1}{\alpha}D_h(z, y) + g(z)\\
& = E(z) + \left(\dfrac{1}{\alpha} - R_f \right)D_h(z, y) \geq E(z) +
\left(\dfrac{1}{\alpha} - R_f \right)\dfrac{\sigma}{2}\|z-y\|^2,
\end{align*}
where the second inequality is from \cref{relatively-bounded} and the last inequality
follows from \cref{strongconvex}. 
\end{proof}
\begin{remark}
\cref{lem:BPGsuffdesc} 
shows that the non-restart condition \cref{criterion:line_search} and the
linear search condition \cref{condition:linesearch} are satisfied when 
\begin{equation}\label{alpha:range}
0<\alpha<\bar\alpha:=\min\left(\frac{1}{2c/\sigma+R_f},\frac{1}{2\eta/\sigma+R_f}\right) \mbox{ and } 0<\alpha_{\min}\leq\bar\alpha
\end{equation}
Therefore, the line search in \cref{alg:esalpha} stops in finite iterations, and thus the \cref{alg:ABPG} is  well defined.
\end{remark}
In the following analysis, we always assume that the parameter $\alpha_{\min}$
satisfies \cref{alpha:range} for simplicity. Therefore, we can obtain the
sufficient decrease property of the sequence generated by \cref{alg:ABPG}.
\begin{corollary}\label{coro:suff}
Suppose the  \cref{assum1} and \cref{assum2} hold. Let $\{x^k\}$ be the sequence generated by the \cref{alg:ABPG}. 
Then, $\{x^k\}\subset\mathcal{M}(x^0)$ and 
\begin{equation}\label{ABPGsuffDec}
E(x^k)-E(x^{k+1}) \geq c_0 \|x^k-x^{k+1}\|^2,
\end{equation}
where $c_0=\min(c,\eta)$.
\end{corollary}
The proof of \cref{coro:suff} is a straightforward result as AA-BPG algorithm is well
defined and the condition \cref{criterion:line_search} or \cref{condition:linesearch} holds at each iteration. 
Let $ \calB(x^0)  $  be the closed ball that  contains $ \calM(x^0) $. Since $ h, F\in C^2$,  there exist $\rho_h,\rho_f>0$ such that
\begin{align}
\rho_h = \sup_{x\in\calB(x^0)} \|\nabla^2 h(x)\|, \quad \rho_f =  \sup_{x\in\calB(x^0)} \|\nabla^2 f(x)\|.
\label{rho_hF}
\end{align}
Thus, we can show the subgradient of each step generated by \cref{alg:ABPG} is bounded by the movement of $x^k$.
\begin{lemma}[Bounded the subgradient]\label{lem:gradbound}
Suppose \cref{assum1} and \cref{assum2} holds. Let $\{x^k\}$ be the sequence
generated by \cref{alg:ABPG}. Then, there exists $c_1=\rho_f+\rho_h/\alpha_{\min}>0$ such that
\begin{align}
\dist(\vzero,\partial E(x^{k+1}))\leq  c_1  (\|x^{k+1} - x^k\| + \bar w\|x^k - x^{k-1}\|),
\label{gradbound}
\end{align}
where  $ \dist(\vzero,\partial E(x^{k+1})) = \inf\{\|y\|: y\in\partial E(x^{k+1})\} $
and $ \rho_h$, $ \rho_f $ are defined as \cref{rho_hF} and $\bar w, \alpha_{\min}$ are constants defined 
in \cref{alg:ABPG} and \cref{alg:esalpha}, respectively.
\end{lemma}

\begin{proof}
By the first order optimality condition of \cref{Iter:ABreg}, we get
\begin{align*}
&\vzero \in \nabla f(y^k)  +\dfrac{1}{\alpha_k}\left(\nabla h(x^{k+1})  - \nabla h(y^k)\right) +  \partial g(x^{k+1})\\
\iff & - \nabla f(y^k) -\dfrac{1}{\alpha_k}\left(\nabla h(x^{k+1})  - \nabla h(y^k)\right)\in \partial g(x^{k+1})
\end{align*}
Since $ f\in C^2 $, we know\,{\cite[Theorem 5.38]{vanconvex}}
\begin{equation}\label{sum_subgrad}
\partial E(x) = \nabla f(x) + \partial g(x).
\end{equation}
From \cref{lem:BPGsuffdesc}, we have $ x^{k},x^{k-1}\in \calM(x^0)$, then $ y^k = (1+w_k)x^k - w_kx^{k-1}\in \calB(x^0) $.  
Together with \eqref{sum_subgrad}, we have
\begin{align*}
\dist(\vzero,\partial E(x^{k+1})) & = \inf_{y\in\partial g(x^{k+1})} \|  \nabla f(x^{k+1})  +  y\|\\
&\leq  \| \nabla f(x^{k+1}) - \nabla f(y^k) -  \dfrac{1}{\alpha_k}\left(\nabla h(x^{k+1})  - \nabla h(y^k)\right)\|\\
& \leq  \| \nabla f(x^{k+1}) - \nabla f(y^k)\| + \dfrac{1}{\alpha_k}\|\nabla h(x^{k+1})  - \nabla h(y^k)\|\\
&\leq (\rho_f+ \dfrac{ \rho_h}{\alpha_k})  \|x^{k+1} - y^k\|\\
&\leq   c_1  (\|x^{k+1} - x^k\| + \bar w\|x^k - x^{k-1}\|),
\end{align*}
where the last inequality is from $y^k = x^k+w_k(x^k-x^{k-1})$ and $w^k\in[0,\bar w]$.
\end{proof}

Now, we are ready to establish the sub-convergence property of \cref{alg:ABPG}.
\begin{thm}\label{Thm:BAPGconvergence}
Suppose \cref{assum1} and \cref{assum2} hold. Let $\{x^k\}$ be the sequence generated by \cref{alg:ABPG}. 
Then, for any limit point $x^*$ of $\{x^k\}$, we have $\vzero\in\partial E(x^*)$.	
\label{thm:convergence}
\end{thm}
\begin{proof}
From \cref{coro:suff}, we know $\{x^k\}\subset\calM(x^0)\subset\calB(x^0)$ and thus bounded. 
Then, the set of limit points of $\{x^k\}$ is nonempty. For any limit point $x^*$, there exist 
a subsequence $\{x^{k_j}\}$ such that $x^{k_j}\to x^*$ as $j\to\infty$.  
We know $\{E(x^k)\}$ is a decreasing sequence. Together with the fact that $E$ is bounded below, 
there exists some $\bar E$ such that $E(x^k)\to\bar E$ as $k\to\infty$. 
Moverover, it has
\begin{equation}
E(x^0)-\bar E = \lim_{K\to\infty}\sum_{j=0}^{K}\left(E(x^j)-E(x^{j+1})\right)\geq c_0\lim_{K\to\infty}\sum_{j=0}^K\|x^j-x^{j+1}\|^2,
\end{equation}
and implies $\|x^k-x^{k-1}\|\to0$ as $k\to\infty$. As a result,
\begin{equation*}
\lim\limits_{k\to\infty}\|x^k - y^{k-1}\| \leq \lim\limits_{k\to\infty}(\|x^k - x^{k-1}\| + \bar{\omega} \|x^{k-1}-x^{k-2}\|) = 0.
\end{equation*} 
Together with \cref{gradbound}, 
it implies that there exists $u^{k_j}\in\partial g(x^{k_j})$ such that
\begin{equation}\label{sublim}
\lim_{j\to\infty}\|\nabla f(x^{k_j})+u^{k_j}\|=0\Rightarrow \lim_{j\to\infty} u^{k_j} = -\nabla f(x^*),
\end{equation}
as $\nabla f$ is continuous and $x^{k_j}\to x^*$ when $j\to\infty$.

In the next, we prove $\lim\limits_{j\to\infty}g(x^{k_j})= g(x^*)$.
It is easy to know  that  $\lim\limits_{j\to \infty}x^{k_j-p}= x^*$ for finite $ p\geq 0 $ since $ \lim\limits_{k\to\infty}\|x^k - x^{k-1}\| = 0 $. 
Thus,  we have $ y^{k_j-1} = x^{k_j-1} + w_{k_j-1}(x^{k_j-1}-x^{k_j-2})\to x^* $ as $ j\to\infty $.  
From \cref{Iter:ABreg}, we know
\begin{equation}
\begin{split}
&g(x^{k_j}) + \langle x^{k_j}- y^{k_j-1},\nabla f(y^{k_j-1})\rangle + \dfrac{1}{\alpha_k}D(x^{k_j}, y^{k_j-1})\\
\leq~ & g(x) + \langle x- y^{k_j-1},\nabla f(y^{k_j-1})\rangle + \dfrac{1}{\alpha_{k}}D(x, y^{k_j-1}),\quad \forall x.
\end{split}	
\end{equation}
Let $ x = x^* $ and $ j\to \infty  $, we get $ \limsup\limits_{j\to\infty} g(x^{k_j}) \leq g(x^*) $. By the fact that $ g(x) $ is lower semi-continuous, it has $ \lim\limits_{j\to\infty}g(x^{k_j}) = g(x^*) $.

Thus, by the convexity of $g$, we have
\begin{equation}\label{convex}
g(x)\geq g(x^{k_j}) + \langle u^{k_j}, x -x^{k_j}\rangle, \forall x\in\dom g.
\end{equation}
Let $j\to\infty$ in \eqref{convex} and using the facts $x^{k_j}\to x^*$, $g(x^{k_j})\to g(x^*)$ as $j\to\infty$ and \eqref{sublim},  we have $-\nabla f(x^*)\in\partial g(x^*)$ and thus $0\in\partial E(x^*)$.
\end{proof}

Furthermore, the sub-sequence convergence can be strengthen by imposing 
the next assumption on $E$ which is known as the Kurdyka-Lojasiewicz (KL) property\,\cite{bolte2014proximal}. 
\begin{assumption}\label{assum3}
$ E(x) $ is the $ KL $ function, i.e. for all $ \bar{x}\in \dom \partial E  := \{x :
\partial E(x)\neq \emptyset\}$, there exist $\eta>0$, a neighborhood $ U $ of $
\bar x $ and $\psi\in\Psi_\eta: =\{\psi\in C[0,\eta)\cap C^1(0,\eta), \text{ where }
\psi \text{ is concave}, \psi(0)=0, \psi^{'}>0 \text{ on } (0,\eta)\}$ such that for
all $x\in U\cap \{x:E(\bar x)<E(x)<E(\bar x)+\eta\}$, the following inequality holds,
\begin{equation}
\psi^{'}(E(x)-E(\bar x))\,\dist(\vzero,\partial E(x))\geq 1.
\end{equation}
\end{assumption}

\begin{thm}\label{Thm:BAPGSeqconver}
Suppose \cref{assum1}, \cref{assum2} and \cref{assum3} hold. Let $\{x^k\}$ be 
the sequence generated by \cref{alg:ABPG}. Then, there exists a point $x^*\in\calB(x^0)$ such that 
\begin{align}
\lim\limits_{k\rightarrow+\infty}x^k=x^* ,\quad	 \vzero\in\partial E(x^*).
\end{align}	
\end{thm}
\begin{proof}
The proof is in Appendix A.
\end{proof}
It is known from \cite{bolte2014proximal} that many functions satisfy \cref{assum3} including the energy function in PFC models.
In the following context, we apply the AA-BPG method for solving the PFC models \cref{min:finite} by introducing two Bregman distances.

\section{AA-BPG method for solving PFC models}
\label{sec:applicationPFC}

The problem \cref{min:finite} can be reduced to \cref{GeneralFormulation} by setting 
\begin{equation}\label{eq:PFC1}
f(\hPhi) = F(\hPhi),\quad g(\hPhi)=G(\hPhi) + \delta_{\cS}(\hPhi)
\end{equation}
where $\cS = \{\hPhi:e_1^\top \hPhi =0\}$ and $\delta_{\cS}(\hPhi) =0$ if
$\hPhi\in\cS$ and $+\infty$ otherwise.  The main difficulty of applying
\cref{alg:ABPG} is solving the subproblem \cref{BProxsubprob} efficiently. In this
section, two different strongly convex functions $h$ are chosen as 
\begin{equation}\label{Kernal_h}
h(x) = \frac{1}{2}\|x\|^2 ~ \mbox{ (P2) }\quad\mbox{ and }\quad
h(x)=\frac{a}{4}\|x\|^4+\frac{b}{2}\|x\|^2+1 ~\mbox{ (P4), }
\end{equation}
where $a,b>0$ and (P2) and (P4) represent the highest order of the $\ell^2$ norm.

\noindent{\underline{\bf Case (P2).}}
The Bregman distance of $D_h$ is reduced to the Euclidean distance, i.e.\
\begin{equation}
D_h(x,y) = \dfrac{1}{2}\|x-y\|^2.
\end{equation}
The subproblem \cref{BProxsubprob} is reduced to 
\begin{equation}\label{L2:sub}
\min_{\hPhi}~G(\hPhi) + \langle \nabla F(\hPsi^k), \hPhi-\hPsi^k\rangle + \frac{1}{2\alpha_k}\|\hPhi-\hPsi^k\|^2, \mbox{ s.t. } e_1^\top \hPhi = 0,
\end{equation}
where $\hPsi^k = \hPhi^k+w_k(\hPhi^k-\hPhi^{k-1})$.
Although the \cref{L2:sub} is a constrained minimization problem, it has a closed
form solution based on our discretization which leads to a fast computation.
\begin{lemma}\label{lemma:l2}
Given $\alpha_k>0$, if $e_1^\top\hPsi^k=0$, the minimizer  of \cref{L2:sub}, denoted by $\hPhi^{k+1}$, is given by 
	\begin{align}\label{BProx_Norm2}
	\hPhi^{k+1} = \left(I+\alpha_kD\right)^{-1}\left(\hPsi^k  -\alpha_k\calP_1\nabla F(\hPsi^k)\right),
	\end{align}
	where $D$ is defined in \cref{eq:Grad}  and $ \calP_1 = I-e_1e_1^\top $ is the projection into the set $ \calS $.
\end{lemma}
	\begin{proof}
	The KKT conditions for this subproblem \cref{BProxsubprob} can be written as
	\begin{align}
	\label{eq:proxKKT1}
	&\nabla G(\hPhi^{k+1}) + \nabla F(\hPsi^k) + \dfrac{1}{\alpha_k}\left(\hPhi^{k+1} - \hPsi^k\right) - \xi_k e_1 = 0,\\
	&e_1^\top \hPhi^{k+1} = 0,
	\label{eq:proxKKT2}
	\end{align}
	where $ \xi_k $ is the Lagrange multiplier. Taking the inner product with $ e_1 $
	in \cref{eq:proxKKT1},  we obtain
	\begin{align*}
	\xi_k = e_1^\top \left(\nabla G(\hPhi^{k+1}) + \nabla F(\hPhi^k) - \dfrac{1}{\alpha_k}\hPsi^k\right).
	\end{align*}
	Using \cref{eq:proxKKT2} and \cref{eq:Grad}, we know
	\begin{equation*}
	e_1^\top \nabla G(\hPhi^{k+1})= e_1^\top (D\hPhi^{k+1}) = 0.
	\end{equation*}
	 Together with $ e_1^\top \hPsi^k = 0$, we have
$\xi_k = e_1^\top \nabla F(\hPsi^k)$.
	Substituting it into \cref{eq:proxKKT1}, it follows that
	\begin{align*}
	\hPhi^{k+1}
		&=\left(\alpha_k D+ I\right)^{-1}\left(\hPsi^k  -\alpha_k\calP_1\nabla F(\hPsi^k)\right).
	\end{align*}
	\end{proof}
It is noted that from the proof of \cref{lemma:l2}, the feasibility assumption
$e_1^\top \hPsi^k=0$ holds as long as $e_1^\top \hPhi^0=0$ which can be set in the
initialization. The detailed algorithm is given in \cref{alg:adapAPG} with $K=2$.

\noindent{\underline{\bf Case (P4).}} In this case, the subproblem \cref{BProxsubprob} is reduced to
\begin{equation}\label{L4:sub}
\min_{\hPhi} G(\hPhi) + \langle \nabla F(\hPsi^k),\hPhi-\hPsi^k\rangle + D_h(\hPhi,\hPsi^k), \mbox{ s.t. } e_1^\top \hPhi = 0.
\end{equation}
where $\hPsi^k=\hPhi^k+w_k(\hPhi^k-\hPhi^{k-1})$. The next lemma shows the optimal
condition of minimizing \cref{L4:sub}.
\begin{lemma}
Given $\alpha^k>0$. If $e_1^\top\hPsi^k=0$,  the minimizer of \cref{L4:sub}, denoted by $\hPhi^{k+1}$, is given by
	\begin{align}\label{BProx_Norm4}
	\hPhi^{k+1} = [\alpha_k D + (ap^*+b)I]^{-1}(\nabla h(\hPsi^k) - \alpha_k \calP_1\nabla F(\hPsi^k)),	
	\end{align}
	where $D$ is given in \cref{eq:Grad} and $ p^* $ is a fixed point of $  p =\|\hPhi^{k+1}\|^2 :=  r(p) $.
\end{lemma}
\begin{proof}
The KKT conditions of \cref{L4:sub} imply that there exists a Lagrange multiplier $\xi_k$ such that $(\hPhi^{k+1},\xi_k)$ satisfies
	\begin{align}
	\label{KKT1}
	&\alpha_k \nabla G(\hPhi^{k+1}) + \alpha_k \nabla F(\hPsi^k) + \nabla h(\hPhi^{k+1}) - \nabla h(\hPsi^k) -\xi_k e_1= 0,\\
	\label{KKT2}
	& e_1^\top \hPhi^{k+1} = 0.
	\end{align}
	 Since $ e_1^\top \hPhi^k = 0 $  and $ \nabla h(x) = (a\|x\|^2+b)x $, \cref{KKT2}
	 and \cref{eq:Grad} imply
	\begin{align*}
	 e_1^\top \nabla G(\hPhi^{k+1}) = e_1^\top (D\hPhi^{k+1}) = 0,\quad e_1^\top \nabla h(\hPsi^k) = (a\|\hPsi^k\|^2+b)e_1^\top \hPsi^k = 0,
	\end{align*}
	where $D$ is defined in \cref{eq:Grad}.
	Substituting the above equalities into  \cref{KKT1} implies $\xi_k = \alpha_k e_1^\top \nabla F(\hPsi^k)$.
	Denote 
	\begin{align*}
	p:=\|\hPhi^{k+1}\|^2 \geq 0,\quad  \beta :=  \nabla h(\hPsi^k) - \alpha_k \nabla F(\hPsi^k) + \xi e_1 = \nabla h(\hPsi^k) - \alpha_k \calP_1\nabla F(\hPsi^k).
	\end{align*}
	From \cref{KKT1}, we obtain a fixed point problem with respect to $ p $
	\begin{align}\label{fixPoint}
	p=\|\hPhi^{k+1}\|^2  = \|[D+(ap+b)I]^{-1}\beta\|^2 := r(p).
	\end{align}
	Let $ R(p) = r(p) - p $. Then $ R(0) = \|(D+bI)^{-1}\beta\|^2 \geq 0 $, $R(p)\to-\infty$ as $p\to\infty$ and 
	\begin{align*} R'(p) = -2a\sum_{i=1}^n\dfrac{\beta_i^2}{(D_{ii} + ap + b)^3} - 1 <0, \quad \forall p \geq 0 ,
	\end{align*}
	there is an unique zero $ p^*\geq 0 $ of $ R(p) $, i.e. $ p^* = r(p^*)  $. Thus,	
	\begin{align*}
	\hPhi^{k+1} = [\alpha_k D+(ap^*+b)I]^{-1}(\nabla h(\hPsi^k) - \alpha_k \calP_1\nabla F(\hPsi^k)).
	\end{align*}
\end{proof}
It is noted that the fixed point equation \cref{fixPoint} is a nonlinear scalar
equation which can efficiently solved by many existing solvers. The detailed algorithm is given in \cref{alg:adapAPG} with $K=4$.

\begin{algorithm}[!pbht]
	\caption{AA-BPG-K method for PFC model}
	\label{alg:adapAPG}
	\begin{algorithmic}[1]
		\REQUIRE $\hPhi^1=\hPhi^0$, $ \alpha_0 >0 $, $w_0\in [0,1]$, $\rho\in (0,1)$, $ \eta,c,\bar w>0$ and $k=1$.
		\WHILE { stop criterion is not satisfied}
		\STATE Update $\hPsi^k = \hPhi^k - w_k(\hPhi^k-\hPhi^{k-1})$
		\STATE Estimate $ \alpha_k $ by \cref{alg:esalpha}
		\IF {$K=2$}
    	\STATE Calculate $z^k=\left(\alpha_k D+ I\right)^{-1}\left(\hPsi^k  -\alpha_k\calP_1\nabla F(\hPsi^k)\right)$
    	\ELSIF {$K=4$}
    	\STATE Calculate the fixed point of \cref{fixPoint}.
    	\STATE Calculate  $z^k=	[\alpha_k D + (ap^*+b)I]^{-1}(\nabla h(\hPsi^k) - \alpha_k \calP_1\nabla F(\hPsi^k))$
    	\ENDIF
		\IF {$E(\hPhi^k)-E(z^{k})\geq c\|\hPhi^k-z^{k}\|^2$ }
		\STATE $ \Phi^{k+1} = z^k $ and update $w_{k+1}\in [0,\bar w]$.
		\ELSE
		\STATE $ \Phi^{k+1} = \Phi^k $ and reset $w_{k+1}=0$.
		\ENDIF	
		\STATE $ k =k+1 $.
		\ENDWHILE	
	\end{algorithmic}
\end{algorithm}

\subsection{Convergence analysis for \cref{alg:adapAPG}}
The convergence analysis can be directly applied for \cref{alg:adapAPG} if the assumptions 
required in \cref{thm:convergence} hold. We first show that the energy function $E$ in PFC model 
satisfies \cref{assum1} and \cref{assum3}. Then,  \cref{assum2} is analyzed for Case (P2) and Case (P4) independently. 
\begin{lemma}\label{PFCproperty1}
	Let $ E_0 = F(\hPhi) + G(\hPhi) $ and $ E(\hPhi)  = E_0(\hPhi)+\delta_{\mathcal{S}}(\hPhi)$ be the energy
	functional which is defined in \cref{eq:PFC1}.  Then, it satisfies
	\begin{enumerate}
		\item $E$ is bounded below and the sub-level set $\calM(\hPhi^0)$ is compact for any $\hPhi^0\in\cS$.
		\item $E$ is a KL function, and thus satisfies \cref{assum3}.
	\end{enumerate}
\end{lemma}
\begin{proof}
	From the continuity and the coercive property of $ F $, i.e. $ F(\hPhi)\to +\infty $ as $ \hPhi\to \infty $,
	the sub-level set $ \calS_0:=  \{\hPhi:E_0(\hPhi) \leq E_0(\hPhi^0)\} $ is compact for any $ \hPhi^0 $. 
	Together with $ \calS $ is closed, it follows that $ \calM(\hPhi^0) = \calS\cap\calS_0 $ is compact for any $ \hPhi^0 $.
	
	Moreover, according to Example 2 in\,\cite{bolte2014proximal}, it is easy to know that $ E(\hPhi)  $ is semi-algebraic function, then it is KL function by Theorem 2 in \,\cite{bolte2014proximal}.
\end{proof}
\begin{lemma}\label{PFCproperty2}
	Let $F(\hPhi)$ be defined in \cref{min:finite}. Then, we have 
	\begin{enumerate}
		\item If $h$ is chosen as (P2) in \cref{Kernal_h}, then $F$ is relative smooth with respect to $h$ in $\calM$ for any compact set $\calM$.
		\item If $h$ is chosen as (P4) in \cref{Kernal_h}, then $F$ is relative smooth with respect to $h$.
	\end{enumerate}
\end{lemma}
\begin{proof}
	Denote $\hPhi^{\otimes k}:= \hPhi\otimes\hPhi\otimes\cdots\otimes\hPhi$
	where $\otimes$ is the tensor product. Then, $F(\hPhi)$ is the $4th$-degree polynomial,
	i.e. $F(\hPhi) = \sum_{k=2}^4\langle\mathcal{A}_k,\hPhi^{\otimes k}\rangle$ where the $kth$-degree monomials are arranged as a $kth$-order tensor $\mathcal{A}_k$.
	For any compact set $\calM$, $\nabla^2 F$ is bounded and thus $F$ is
	relative smooth with respect to any polynomial function in $\calM$ which
	includes case (P2). When $h$ is chosen as (P4), according to Lemma 2.1
	in\,\cite{Bregman}, there exists $R_F>0$ such that $F(\hPhi)$ is $R_F$-relative smooth with respect to $ h(x) $.
\end{proof}
Combining \cref{PFCproperty1}, \cref{PFCproperty2} with \cref{Thm:BAPGSeqconver}, we can directly establish the convergence of \cref{alg:adapAPG} .
\begin{thm}\label{thm:adapAPG}
	Let $ E(\hPhi)  = F(\hPhi) + G(\hPhi)+\delta_{\mathcal{S}}(\hPhi)$ be the energy
	function which is defined in \cref{eq:PFC1}. The following results hold.
	\begin{enumerate}
		\item Let $\{\hPhi^k\}$ be the sequence generated by \cref{alg:adapAPG} with $K=2$. If $\{\hPhi^k\}$ is bounded,  
		then $\{\hPhi^k\}$ converges to some $\hPhi^*$  and $0\in\partial E(\hPhi^*)$.
		\item Let $\{\hPhi^k\}$ be the sequence generated by \cref{alg:adapAPG} with $K=4$. 
		Then, $\{\hPhi^k\}$ converges to some $\hPhi^*$  and $0\in\partial E(\hPhi^*)$.
	\end{enumerate}
\end{thm}
It is noted that when $h$ is chosen as (P2), we cannot bounded the growth of $F$ as
$F$ is a fourth order polynomial. Thus, the boundedness assumption of $\{\hPhi^k\}$
is imposed which is similar to the requirement in the semi-implicit
scheme\,\cite{shen2010numerical}.

\section{Newton-PCG method}\label{sec:second-order}
Despite the fast initial convergence speed of the gradient based methods, the tail
convergence speed becomes slow. Therefore it can be further locally accelerated by the
feature of Hessian based methods.
In this section, we design a practical Newton method to solve the PFC models
\cref{min:finite} and provide a hybrid accelerated framework. 
\subsection{Our method}
Define $ Z := [0,I_{N-1}]^\top$, any vector $ \hPhi $ that satisfies the constraint $
e_1^\top \hPhi  = 0$ has the form	of $ \hPhi = ZU $ with $ U\in \mathbb{C}^{N-1} $.
Since $ Z^\top Z =  I_{ N-1}$, we can also obtain $ U $ from $ \hPhi $ by $ U =
Z^\top \hPhi $. Therefore, the problem \cref{min:finite} is equivalent to
\begin{align}
\min_{U\in \mathbb{C}^{N-1} } E(ZU) = G(ZU) + F(ZU).
\label{pro:unconstrain}
\end{align}
Let $\tE(U) := E(ZU),\, \tG(U) := E(ZU), \,\tF(U) := F(ZU) $,
we have the following facts
\begin{equation}\label{tE}
\begin{aligned}
&\tg :=\nabla \tE(U) = Z^\top \nabla E(ZU) = Z^\top g,
\\
&\tJ:=\nabla^2 \tE(U) = Z^\top \nabla^2 E(ZU)Z = Z^\top \calJ Z,
\end{aligned}
\end{equation}
where $g=\nabla E(ZU)$ and $\calJ=\nabla^2 E(ZU)$.
Therefore, finding the steady states of PFC models is equivalent to solving the nonlinear equations
\begin{equation}
\nabla \tE(U) = \bzero.
\end{equation}
Due to the non-convexity of $\tE(U)$, the Hessian matrix $\tJ$ may not be 
positive definite and thus a regularized Newton method is applied. 

\noindent{\underline{\bf Computing the Newton direction.}} 
Denote $ \tJ_k := \nabla^2 \tE(U^k) $ and $ \tg_k :=\nabla \tE(U^k)$, 
we find the approximated Newton direction $d_k$ by solving
\begin{equation}\label{NewtonEq}
(\tJ_k + \mu_k I)d_k = -\tg_{k},
\end{equation}
where regularized parameter $ \mu_k $ is chosen  as
\begin{equation}\label{def:mu}
  -c_1 \min\{0, \lambda_{\min} (\tJ_k) \} + c_2 \|\tg_k\|\leq \mu_k\leq \bar{\mu} < +\infty  ~(c_1\geq 1,c_2>0).
\end{equation}
Thus, \cref{NewtonEq} is symmetric, positive definite linear system. 
To accelerate the convergence, an  preconditioned conjugate gradient (PCG) method is adopted.
More specifically, in $k$-th step, we terminate the PCG iterates 
whenever $\|(\tJ_k + \mu_k I)d_k +\tg_{k}\|\leq\eta_k$ in which $\eta_k$ is set as
\begin{align}\label{PCGtol}
\eta_k = \tau\min\{1, \|\tg_{k}\|\},\quad  0<\tau <1,
\end{align}
and the preconditioner $M_k$ is adaptively obtained by setting 
\begin{align}\label{PCG}
M_k = Z(H_k + \mu_k I)^{-1}Z^\top\quad\mbox{with}\quad H_k = D +\delta_k I 
\end{align}
where $D$ is from \cref{eq:Grad} and some $\delta_k>0$.  Let $A=\tJ_k+\mu_k I $,
$b=-\tg_k$ and $M=M_k$, the  PCG method is given in \cref{alg:PCG}, where $
\|x\|_A:= \langle x, Ax\rangle $.
\begin{algorithm}[!pbht]
	\caption{ PCG($\eta$) for solving $ Ax = b $.}
	\label{alg:PCG}
	\begin{algorithmic}[1]
		\REQUIRE $ A, b, \eta, k_{max}$, preconditioner $ M $.
		\STATE Set $ x^0 =0, r_0 =Ax^0- b = -b, p_0 = -M^{-1}r_0, i=0$.
		\WHILE{$ \|r_i\| > \eta $ or $ i<k_{max} $}
		\STATE $ \alpha_{i+1} = \dfrac{\|r_i\|_{M^{-1}}^2}{\| p_i\|_{A}^2}$
		\STATE $ x^{i+1} = x^i + \alpha_{i+1} p_i $
		\STATE $ r_{i+1} = r_i +\alpha_{i+1}Ap_i $
		\STATE $ \beta_{i+1} = \dfrac{\|r_{i+1}\|_{M^{-1}}^2}{\|r_i\|_{M^{-1}}^2} $
		\STATE $ p_{i+1} = -M^{-1}r_{i+1}+\beta_{i+1}p_{i} $
		\STATE $ i = i+1 $
		\ENDWHILE
	\end{algorithmic}
\end{algorithm}

\noindent{\underline{\bf Computing the step size $t_k$.}} Once the Newton direction $d_k$ is obtained, 
the line search technique is applied for finding an appropriate step size $t_k$ that 
 satisfies the following inequality:
\begin{align}
\tE(U^k + t_k d_k) \leq \tE(U^k) + \nu t_k \la \tg_k, d_k\ra,~0<\nu<1.
\label{ieq:line-search}
\end{align}
The existence of $ t_k >0$ that satisfies \cref{ieq:line-search} is given in  \cref{lemma:bound-t_k}.
Then, $ U^{k+1} $ is updated by $U^{k+1} =  U^{k} + t_kd_k$. Our proposed  algorithm is summarized in  \cref{alg:NewtonPCG}. 
\begin{algorithm}[!pbht]
	\caption{Newton-PCG method}
	\begin{algorithmic}[1]
		\REQUIRE$ U^0,\varepsilon,\bar{\mu}, c_1\geq 1, c_2>0, 0<\nu,\rho,\tau <1;$		
		\STATE$  k =0, \tg_0 = \nabla \tE(U^0)$;
		\WHILE {stop criterion is not satisfied}
		\STATE Choose $-c_1 \min\{0, \lambda_{\min} (\tJ_k) \} + c_2 \|\tg_k\|\leq \mu_k\leq \bar{\mu} $;
		\STATE Update $\eta_k = \tau\min(1,\|\tg_k\|)$.
		\STATE Find direction $ d_k $ by solving \cref{NewtonEq} via PCG($\eta_k$) using \cref{alg:PCG};
		\FOR{$ n = 0,1,2,\cdots $}
		\STATE $ t_k = \rho^n; $
		\IF {$ \tE(U^k + t_k d_k) \leq \tE(U^k) + \nu t_k\la \tg_k,d_k\ra $}
		\STATE \textbf{Break};
		\ENDIF
		\ENDFOR
		\STATE $ U^{k+1} = U^k + t_kd_k$;
		\STATE $ k = k+1; $
		\ENDWHILE
	\end{algorithmic}
	\label{alg:NewtonPCG}
\end{algorithm}

\subsection{Convergence analysis for \cref{alg:NewtonPCG}}
We first establish several properties related to the direction $d_k$ computed by the PCG method. 
\begin{lemma}	\label{Them:innerprod} 
	Consider a linear system $Ax=b$ where $A$ is symmetric and positive definite. Let $\{x^i\}$ be the sequence generated by \cref{alg:PCG}, it satisfies
	\begin{align}
	\dfrac{1}{\lambda_{\max}(A)} \leq \dfrac{\la x^i, b\ra}{\|b\|^2}\leq \dfrac{1}{\lambda_{\min}(A)},\quad \forall  i =1,2,\cdots.
	\end{align}
\end{lemma}
\begin{proof}
	The proof is in Appendix B.
\end{proof}
Then, we know the $d_k$ is a descent direction from the next lemma.
\begin{lemma}[Descent direction]\label{lem:4.5}
	Let $d_k$ be generated by PCG($\eta_k$) method (\cref{alg:PCG}). If $\|\tg_k\| >0$, then we have 
	\begin{align}
	-\la d_k, \tg_k \ra \geq l_k:=\dfrac{\|\tg_k\|^2}{\lambda_{\max}(\tJ_{k} + \mu_k I)}\quad\mbox{and}\quad \|d_k\|\leq \bar{d} := \dfrac{\tau + 1}{c_2},
	\label{bound2}
	\end{align}
	where $\tau$, $K$, and $ c_1,c_2 $ are defined in 
	\cref{PCGtol}, \cref{Hbound} and \cref{def:mu}, respectively.
\end{lemma}
\begin{proof}
	The first inequality is a direct consequence of \cref{Them:innerprod}.
	Moreover, let $r_k = (\tJ_k + \mu_k I)d_k + \tg_k $. By \cref{alg:PCG} and \cref{PCGtol}, we have $\|r_k\|\leq \eta_k\leq \tau\|\tg_k\|$. Then,  
	\begin{align*}
	\|d_k\|& = \|(\tJ_k + \mu_k I)^{-1}(r_k-\tg_k)\|\leq \dfrac{\|r_k-\tg_k\|}{\lambda_{\min}(\tJ_k + \mu_k I)}\leq \dfrac{\|r_k\| + \|g_k\|}{c_2\|g_k\|}\leq \dfrac{\tau + 1}{c_2},
	\end{align*}
	where the second inequality is from \cref{def:mu}.
\end{proof}
\begin{lemma}[Lower bound of $ t_k $]
	\label{lemma:bound-t_k} Let $d_k$ be generated by PCG($\eta_k$) method (\cref{alg:PCG}). If $ \|\tg_k \|\geq \varepsilon>0$, then for any $\nu\in(0,1)$, there exists $M_k>0$ and  \begin{align}
	t_{\max}^k : =\min \left\{
	\frac{2(1-\nu)l_k}{M_k\bar{d}^2},1\right\}.
	\label{tmin}
	\end{align}
	such that the inequality \cref{ieq:line-search} holds for $t_k\in(0,t_{\max}^k]$
	where $l_k$ is defined in \cref{bound2}.
\end{lemma}
\begin{proof}
	By the Taylor expansion, we have
	\begin{equation}\label{in1}
	\tE (U^k + td_k) = \tE (U^k) + t\la \tg_k, d_k\ra + \dfrac{t^2}{2}\la d_k, \nabla^2 \tE (\xi^t) d_k\ra,
	\end{equation}
	where $\xi^t\in\mathcal{V}_k=\{V|V=U^k+td_k, t\in[0,1]\}$. As $\tE\in C^2$, there
	exists $M_k>0$ such that $M_k=\sup\{\|\nabla^2 \tE(V)\|\}|V\in\mathcal{V}_k\}$.
	Then, \cref{in1} and \cref{bound2} imply
	\begin{align}\label{in2}
	\tE (U^k + td_k) 
	&\leq  \tE (U^k) + \nu t\la \tg_k, d_k\ra -(1-\nu)l_kt + \dfrac{M_k \bar d^2 }{2}t^2.
	\end{align}
	Define $Q(t)=(1-\nu)l_kt - \dfrac{M_k \bar d^2 }{2}t^2$, we know $Q(t)\geq0$ for
	all $t\in[0,\dfrac{2(1-\nu)l_k}{M_k\bar d^2}]$ which implies \cref{ieq:line-search} holds for all $t\in(0,t_{\max}^k]$.
\end{proof}

\begin{theorem}\label{them: Newton-conver}
	Let $ \tE $ be defined in \cref{pro:unconstrain} and  $\{U^k\}$ be the infinite sequence generated by \cref{alg:NewtonPCG}. 
	Then $\{U^k\}$ is bounded and has the following property:
	\begin{align}\label{limits}
	\lim\limits_{k\rightarrow+\infty} \|\tg_k \| = 0.
	\end{align}
\end{theorem}
\begin{proof}
	Due to the continuity of $ \tF $, $ \tG $ in \cref{pro:unconstrain} and the
	coercive property of $ \tF $, the sublevel set $ \calM_0 = \{U:\tE (U)\leq \tE
	(U^0)\} $ is compact for any $ U^0 $. By the inequality \cref{ieq:line-search}, 
	it is easy to know $\{\tE(U^k)\}$ is a decreasing sequence, and thus $\{U^k\}\subset\calM_0$ a
	nd there exists some $\bar E$ such that $\tE(U^k)\to\bar E$ as $k\to\infty$. 
	 Moreover, from \cref{bound2} and $t_k\in(0,1]$, we know there exist a compact set $\calB_0$ 
	 such that $\{U^k+td^k|t\in(0,1]\}\subset\calB_0$ and thus there exists $ M>0 $ such that 
	\begin{align}
	\|\nabla^2 \tE(U)\|\leq M,\quad \forall\, U\in\calB_0.
	\label{Hbound}
	\end{align}
	 From the proof of \cref{lemma:bound-t_k}, we know $M_k\leq M$ for all $k$. 
	 Moreover, there exists some $\bar\lambda>0$ such that $\lambda_{\max}(\tJ_k+\mu_k I)\leq\bar\lambda$ for all $k$.
	We prove \cref{limits} by contradiction. Assume  $ \limsup\limits_{k\rightarrow+\infty} \|\tg_k \| = \varepsilon >0$ and define the index set
	\begin{align}
	\calI = \cup_{k=1}^{\infty}\calI_k :=\left\{j\in \mathbb{N}:j\leq k, \|\tg_j\|\geq \varepsilon/2 \right\}.
	\label{def:index}
	\end{align}
	Then, we know $|\calI|=\infty$ where $ |\calI|  $ denotes the number of the elements of $ \calI $. Moreover, for all $j\in\calI$, we know 
	\begin{equation}
	l_j\geq\varepsilon/2\bar\lambda\quad\mbox{ and }\quad t_{\max}^j \geq\bar t=\min\left\{\dfrac{(1-\nu)\varepsilon}{M\bar\lambda\bar d^2},1\right\}.
	\end{equation}
	Thus, $\bar t$ is a uniform lower bound for the step size $t$ at $U^j$ for $j\in\calI$, i.e. $t_j\geq t, \forall j\in\calI$, and we have
	\begin{align}\label{in3}
	\tE(U^0) - \tE(U^{k+1}) &= \sum_{j=0}^k(\tE(U^{j}) - \tE(U^{j+1})) \geq \sum_{j\in\calI_k}(\tE(U^{j}) - \tE(U^{j+1})) \\
	& \geq \sum_{j\in\calI_k}-\nu t_j\la\tg_j,d_j\ra \geq \dfrac{\nu\bar t\varepsilon}{2\bar\lambda}|\calI_k|,
	\end{align}
	Let $k\to\infty$ in $\cref{in3}$, we know $\tE(U^0)-\bar E\geq+\infty$, which leads to a contradiction.
\end{proof}

\subsection{Hybrid acceleration framework} 
\label{subsec:hybrid}

{
Many gradient based methods have a good convergent performance at the beginning,
but often show slow tail convergence near the stationary states. In this case, 
the Newton-like method is a natural choice and has a better convergence speed 
when the iteration is near the stationary states.} 
It is noted that the Hessian based method is sensitive to the initial point. 
A key step of mixing two methods is designing a proper criterion to
determine when to launch the Hessian based method. It is difficult to develop a
perfect strategy for all kinds of PFC models. In our experiments, 
we switch to the Newton-PCG algorithm when one of the following criteria is met
\begin{equation}\label{switch}
 |E(\hPhi^k) - E(\hPhi^{k-1})| < \varepsilon_1 \quad \mbox{or}\quad \|g_k - g_{k-1}\| < \varepsilon_2,
\end{equation}
where $\varepsilon_1,\varepsilon_2>0$. Our proposed hybrid accelerated framework is summarized in \cref{alg:hybrid}. 
The M method stands for a certain existing method, such as our AA-BPG method.
    	\begin{algorithm}[!pbht]
		\caption{Hybrid acceleration framework (N-M method)}
		\begin{algorithmic}[1]
			\REQUIRE$ \Phi^0,\varepsilon_1, \varepsilon_2$ and $ k = 0$.
			\WHILE { stop criterion is not satisfied}
			\IF {switching condition is satisfed}
			\STATE Perform Newton-PCG method (\cref{alg:NewtonPCG}); 
			\ELSE
			\STATE {Perform M method};
			\ENDIF
			\STATE $ k = k+1 $;
			\ENDWHILE
		\end{algorithmic}
		\label{alg:hybrid}
	\end{algorithm}

\begin{remark}
{
The idea of hybrid method provides a general framework for local acceleration. 
Our Newton-PCG methods can not only combines with the AA-BPG methods, 
but also with many existing  methods. It's worth noting that directly using 
the Newton-PCG method may converge to a bad stationary point or lead to slow convergence since the initial point is not good. 
}
\end{remark}

\section{Numerical results}\label{sec:result}

In this section, we present several numerical examples for our proposed methods and
compare the efficiency and accuracy with existing methods. 
Our approaches contain AA-BPG-2 and AA-BPG-4 (see \cref{alg:adapAPG}), and hybrid
method (see \cref{alg:NewtonPCG}), and the comparison
methods\,\cite{xu2006stability, shen2010numerical, yang2016linear, shen2019new} include the first-order temporal accuracy
semi-implicit scheme (SIS), the first-order temporal accuracy stabilized semi-implicit scheme (SSIS1), 
the second-order temporal accuracy stabilized semi-implicit scheme (SSIS2), the invariant energy
quadrature (IEQ) and scalar auxiliary variable (SAV) approaches.
All methods are employed to calculate the stationary states of finite dimensional PFC models,
including the LB model for periodic crystals and the LP model for quasicrystals.  
Note that these methods all guarantee mass conservation.
The step size $\alpha_k$ in our approaches are obtained adaptively by the linear
search technique, while the fixed step size $\alpha$ of others are chosen to guarantee the best
performance on the premise of energy dissipation.
In efficient implementation of the Newton-PCG method, 
the parameters in \cref{PCGtol} and \cref{PCG} 
are set with $\tau=0.01$, $\delta_k=0.7\max \Lambda_k^{(')}$, and
 $\mu_k$ is chosen as\,\cite{xiao2018regularized}. To show the energy
tendency obviously, we calculate a reference energy $E_s$ by choosing the invariant energy value 
as the grid size converges to 0. From our numerical tests, the reference energy has 14 significant decimal digits.
All experiments were performed on a workstation with a 3.20 GHz
CPU (i7-8700, 12 processors). All code were written by MATLAB
language without parallel implementation. 

\subsection{AA-BPG method}
\subsubsection{Periodic crystals}
For the LB model, we use three dimensional periodic crystals of the double gyroid and
the sigma phase, as shown in \cref{fig:crystals}, to
demonstrate the performance of our approach. In the hybrid method of
\cref{alg:NewtonPCG}, we choose the gradient difference {$\|g_k - g_{k-1}\| <
10^{-3}$  } as the measurement to launch the Newton-PCG algorithm.
\begin{figure}[!htbp]
	\centering
	\includegraphics[width=4.0in]{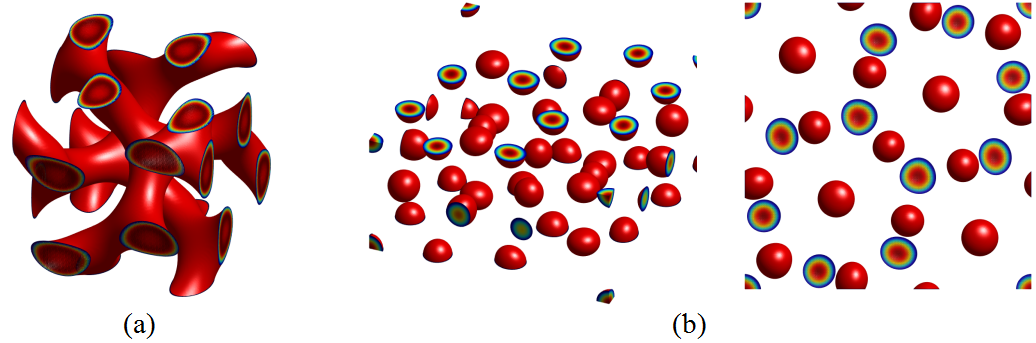}
	\caption{The stationary periodic crystals in LB model.
	(a) Double gyroid phase with $ \xi = 0.1,\tau = -2.0, \gamma = 2.0 $.
	(b) Sigma phase with $ \xi = 0.1,\tau = 0.01, \gamma = -2.0$ from two perspectives.
	}
	\label{fig:crystals}
\end{figure}

\paragraph{Double gyroid}
The double gyroid phase is a continuous network periodic phase
whose initial values can be chosen as 
\begin{align}
\phi(\br) = \sum_{\bh\in\Lambda_0^{DG}} \hphi(\bh) e^{i (\bB\bh)^\top 
\cdot \br },
\label{eq:LB:initial}
\end{align}
where initial lattice points set $\Lambda_0^{DG}\subset\bbZ^3$ only on
which the Fourier coefficients located are nonzero.
The corresponding $\Lambda_0^{DG}$ of the double gyroid phase
can be found in the Table 1 in \cite{jiang2013discovery}.
The double gyroid structure belongs to the cubic crystal system,
therefore, the $3$-order invertible matrix can be chosen as 
$\bB= (1/\sqrt{6})\bI_3$. Correspondingly, the computational
domain in physical space is $\Omega=[0,2\sqrt{6}\pi)^3$.
The parameters in LB model \cref{eq:LB} are set as 
$\xi = 0.1, \tau = -2.0, \gamma = 2.0$.  $128^3$ wavefunctions are used in these simulations.
\cref{fig:crystals} (a) shows the stationary solution of double gyroid profile.

\cref{fig:DG:comparison} gives iteration process of the above-mentioned approaches,
including the relative energy difference and the gradient changes with iteration, and
the CPU time cost. {The reference energy value $E_s = -12.94291551898271 $ is the
finally convergent value.}
As is evident from these results, 
 our AA-BPG methods are most efficient among these approaches under the premise of ensuring energy dissipation.
The AA-BPG-4 and AA-BPG-2 approaches have nearly the same numerical behaviors,
however, the AA-BPG-4 method spends a little more CPU time than AA-BPG-2 scheme does.
The reason is attributed to the cost of solving the subproblem
\cref{BProxsubprob} at each step.
For AA-BPG-2 scheme, \cref{BProxsubprob} can be solved analytically, while 
for AA-BPG-4 method, \cref{BProxsubprob} is required to numerically solve a nonlinear system.
{In \cref{fig:DG:tsteps}, we give the step sizes of AA-BPG-2/4 scheme.}
\begin{figure}[!htbp]
	\centering
	\includegraphics[width=5in]{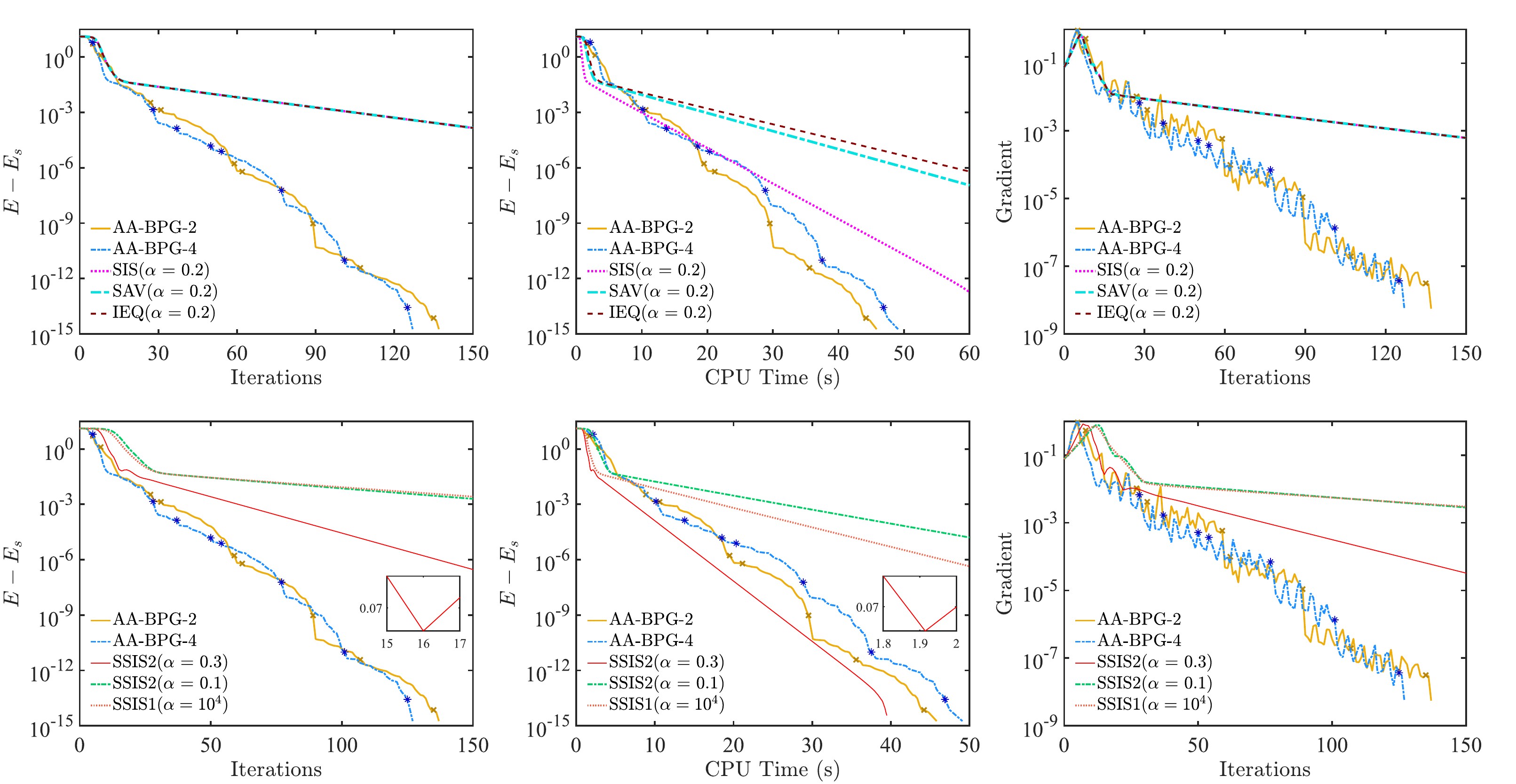}
	\caption{
	Double gyroid phase: comparisons of numerical behaviors of AA-BPG-2/4 approaches with
	\textbf{First row}: SIS, SAV and IEQ;
	\textbf{Second row}: SSIS1 and SSIS2.  
	\textbf{Left column}: Relative energy over iterations; \textbf{Middle column}:
	Relative energy over CPU times; \textbf{Right column}: Gradient over iterations;
	The blue and yellow $ \times $s mark where restarts occurred.
	}
	\label{fig:DG:comparison}
\end{figure}
\begin{figure}[!htbp]
	\centering
	\includegraphics[width=2in]{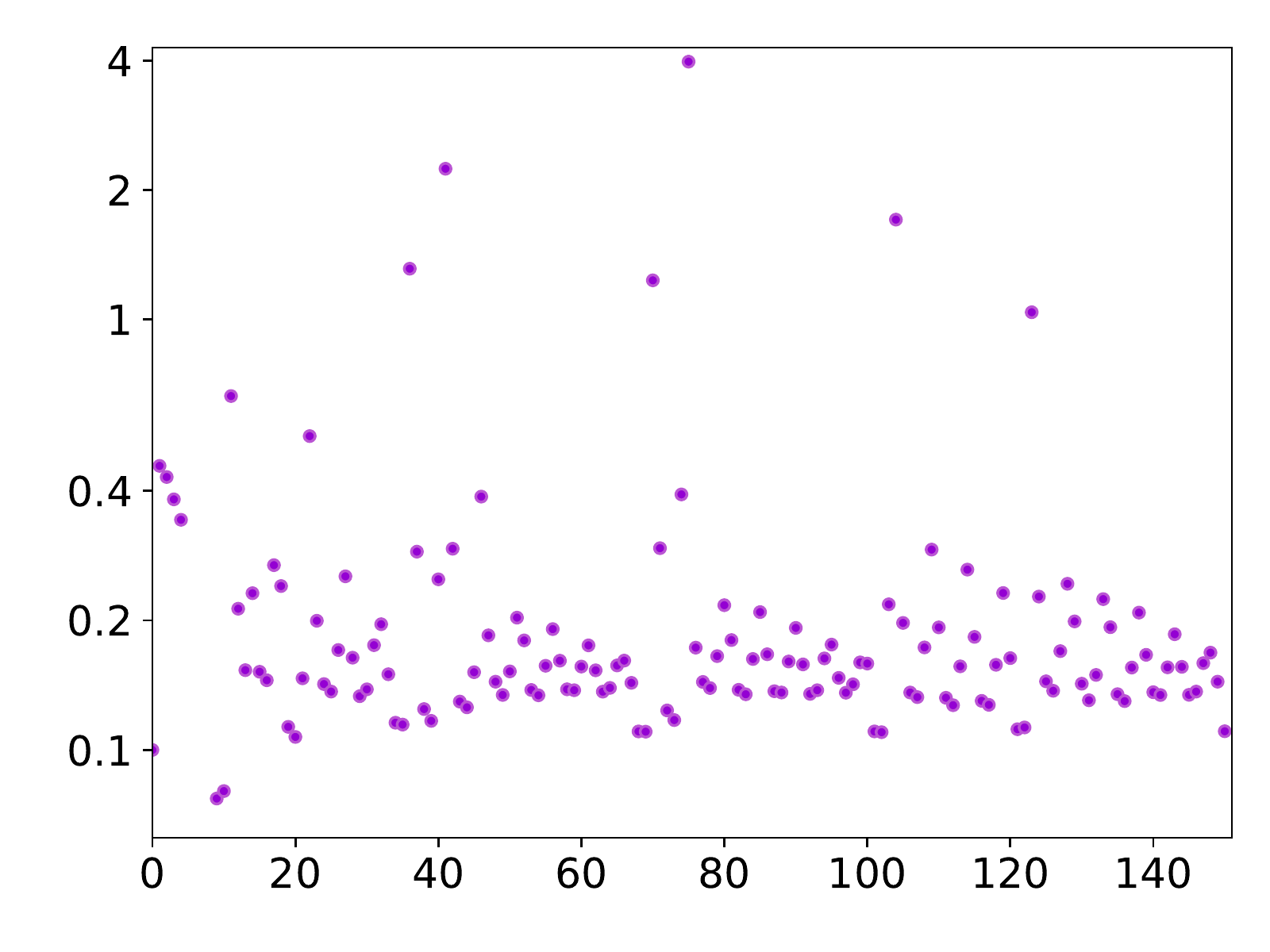}
	\includegraphics[width=2in]{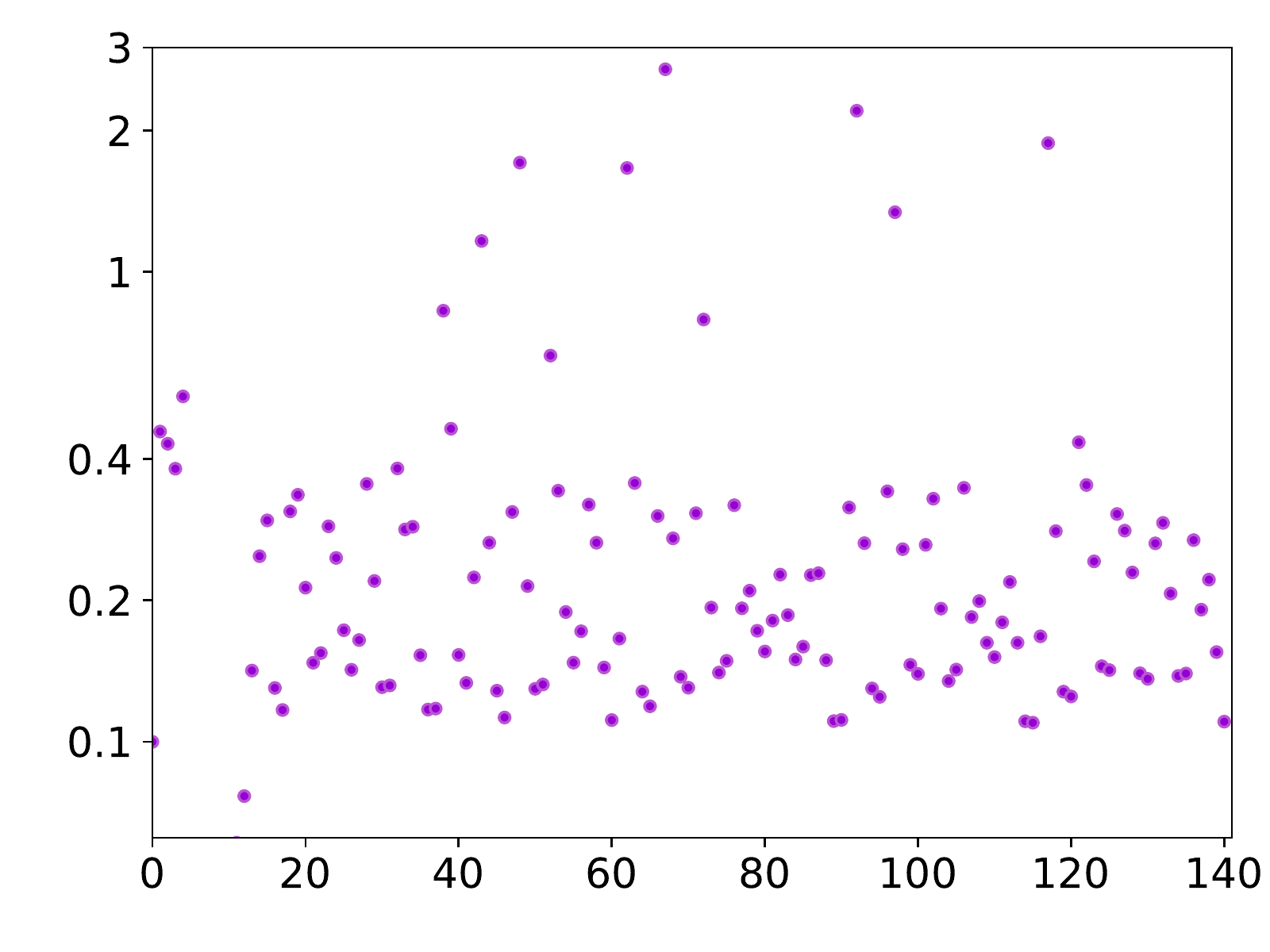}
	\caption{
	Double gyroid phase: the step sizes of \textbf{Left}: AA-BPG-2; \textbf{Right}: AA-BPG-4 approach.
	}
	\label{fig:DG:tsteps}
\end{figure}

The SIS, SAV and IEQ approaches have almost the same iterations. 
Theoretically, the convergence of the SIS is based on the assumption of global
Lipschitz constants, while the SAV method always has a modified energy dissipation 
through adding an arbitrary scaler auxiliary parameter $C$
which guarantees the boundedness of the bulk energy term.
The original energy dissipation property of the SAV method depends on the
selection of $C$. For computing the double gyroid phase,
we find that when $C$ is smaller than $10^{6}$, the SAV scheme cannot keep the
original energy dissipation property even if we adopt a small step size $0.001$. 
Further increasing $C$ to $10^{8}$, we can use a large step size $\alpha=0.2$ to
obtain the original energy dissipation feature. Note that 
there exists a gap between the modified energy and the original energy no matter what
the auxiliary parameters are.  Like in the SAV method, similar results and phenomena
have been also found in the IEQ approach.  Among the three methods, the SIS spends
the fewest CPU times. The reason is that the SAV and IEQ methods requires to solve a
subsystem at each step while the SIS does not.

The SSIS1 is an unconditionally stable scheme through imposing a stabilized term on SIS. 
Its energy law holds under the assumption of the stabilizing parameter greater than
the half of global Lipschitz constant. The step size $\alpha$ can be arbitrary
large while the effective step size has a limit. From the numerical results,
SSIS1 with $\alpha=10^4$ shows a slower convergent rate than the SIS with
$\alpha=0.2$ does.  An interesting scheme is the conditionally stable SSIS2 proposed
in \cite{shen2010numerical} that introduces a center difference stabilizing term to
guarantee the second order temporal accuracy.
From the point of continuity, the SSIS2 actually adds an inertia term onto the original
gradient flow system. The inertia term can accelerate the convergent speed but often
accompanied with some oscillations if the step size is large. 
As \cref{fig:DG:comparison} shows, when $\alpha=0.1$ the SSIS2 has almost the same
convergent speed with the SSIS1, and holds the energy dissipation property. If
increasing $\alpha$, such as $0.3$, the SSIS2 obtains an accelerated speed but with
oscillations.

\paragraph{Sigma phase}
The second periodic structure considered here is the sigma phase,
which is a spherical packed phase recently discovered in
block copolymer experiment\,\cite{lee2010discovery}, and the self-consistent mean-field
simulation\,\cite{xie2014sigma}.
The sigma phase has a larger, much more complicated tetragonal unit cell with
$30$ atoms. For such a pattern, we implement our algorithm 
on bounded computational domain $\Omega=[0, 27.7884)\times [0,
27.7884)\times [0, 14.1514)$. Correspondingly, the
initial values can be found in\,\cite{xie2014sigma}.
When computing the sigma phase, the parameters 
are set as $\xi = 1.0, \tau = 0.01, \gamma = 2.0$ and $256\times 256 \times 128 $
wavefunctions are used to discretize LB energy functional. 
The stationary morphology is shown in \cref{fig:crystals} (b).
As far as we know, it is the first time to find such complicated sigma phase in such
a simple PFC model.

\begin{figure}[htbp]
	\centering
	\includegraphics[width=5.0in]{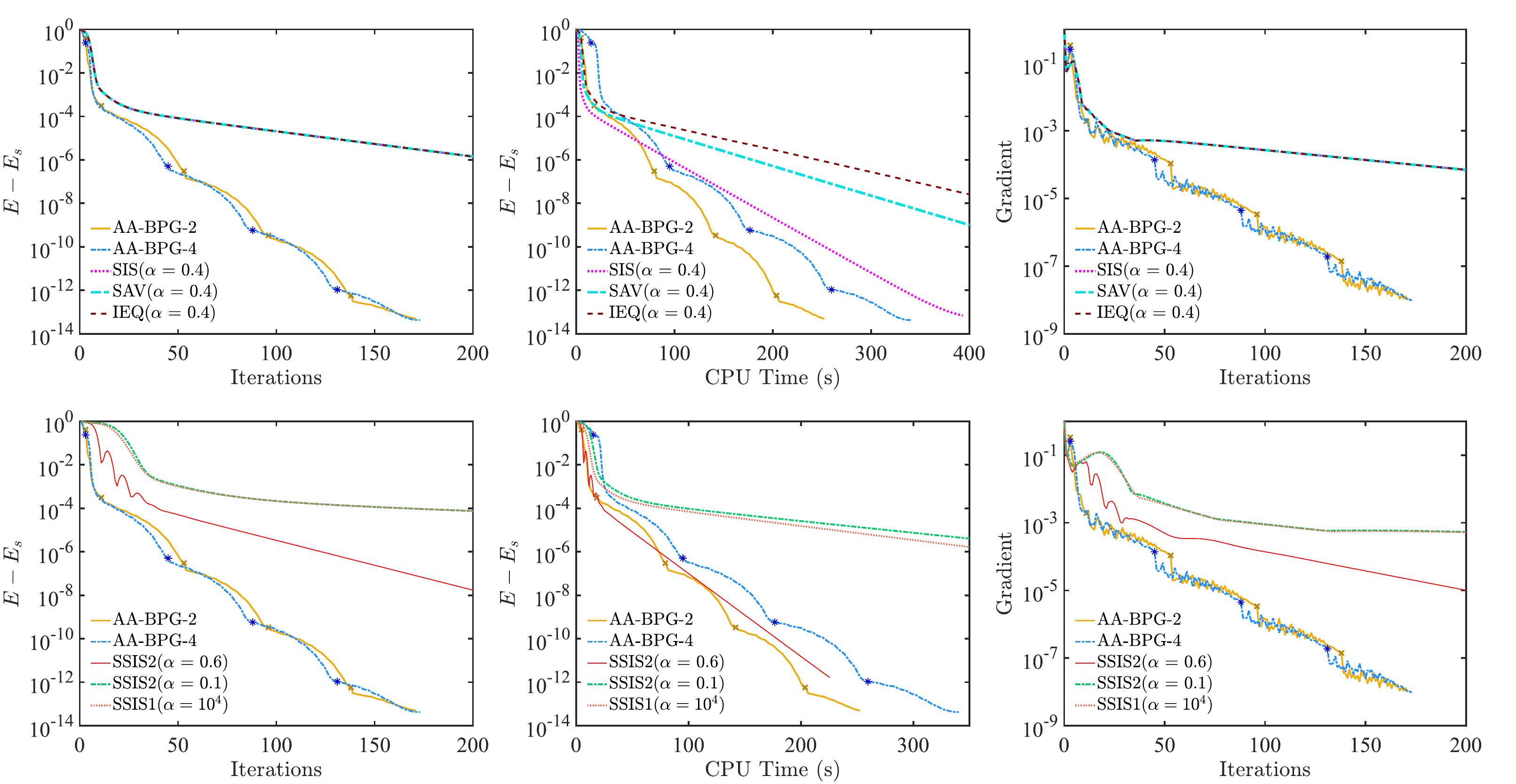}
	\caption{
	Sigma phase: comparisons of numerical behaviors of the
	AA-BPG-2/4 approaches with other numerical methods.
	The information of these plots is the same with \cref{fig:DG:comparison}.
	}
	\label{fig:Sigma:comparison}
\end{figure}

\cref{fig:Sigma:comparison} compares our proposed methods with other numerical schemes. 
We still use the reference energy value{ $ E_s = -0.93081648457086 $ }as the baseline 
to observe the relative energy changes of various numerical approaches.
Again, as shown in these results, on the premise of energy dissipation, 
the new developed  gradient based approaches demonstrate a better performance over the existing
methods in computing the sigma phase. 
Among these methods, the AA-BPG-2 method is the most efficient.

\subsubsection{Quasicrystals}
For the LP free energy \cref{eq:LP}, we take the two dimensional dodecagonal
quasicrystal as an example to examine the performance of our
proposed approach. 
For dodecagonal quasicrystals, two length scales $q_1$
and $q_2$ equal to $1$ and $2\cos(\pi/12)$, respectively.
Two dimensional dodecagonal quasicrystals can be embedded into
four dimensional periodic structures, therefore, the projection
method is carried out in four dimensional space.
The $4$-order invertible matrix $\bB$ associated with to
four dimensional periodic structure is chosen as $\bI_4$. 
The corresponding computational domain in real space is $[0,2\pi)^4$. 
The projection matrix $\calP$ in Eq.\,\cref{eq:pm} of the
dodecagonal quasicrystals is
\begin{equation}
\mathcal{P} =\left(
\begin{array}{cccc}
1 & \cos(\pi/6) & \cos(\pi/3) & 0 \\
0 & \sin(\pi/6) & \sin(\pi/3) & 1 
\end{array}
\right).
\label{eqn:DDQC:projMatrix}
\end{equation}
The initial solution is 
\begin{align}
\phi(\br) = \sum_{\bh\in\Lambda^{QC}_0} \hphi(\bh)
e^{i[(\mathcal{P}\cdot\mathbf{B}\bh)^\top\cdot\br]},
~~\br\in\mathbb{R}^2,
\end{align}
where initial lattice points set $\Lambda_0^{QC}\subset\bbZ^4$ 
can be found in the Table 3 in \cite{jiang2014numerical}
on which the Fourier coefficients $\hphi(\bh)$ located are nonzero.

\begin{figure}[htbp]
		\centering
		\includegraphics[width=3.5in]{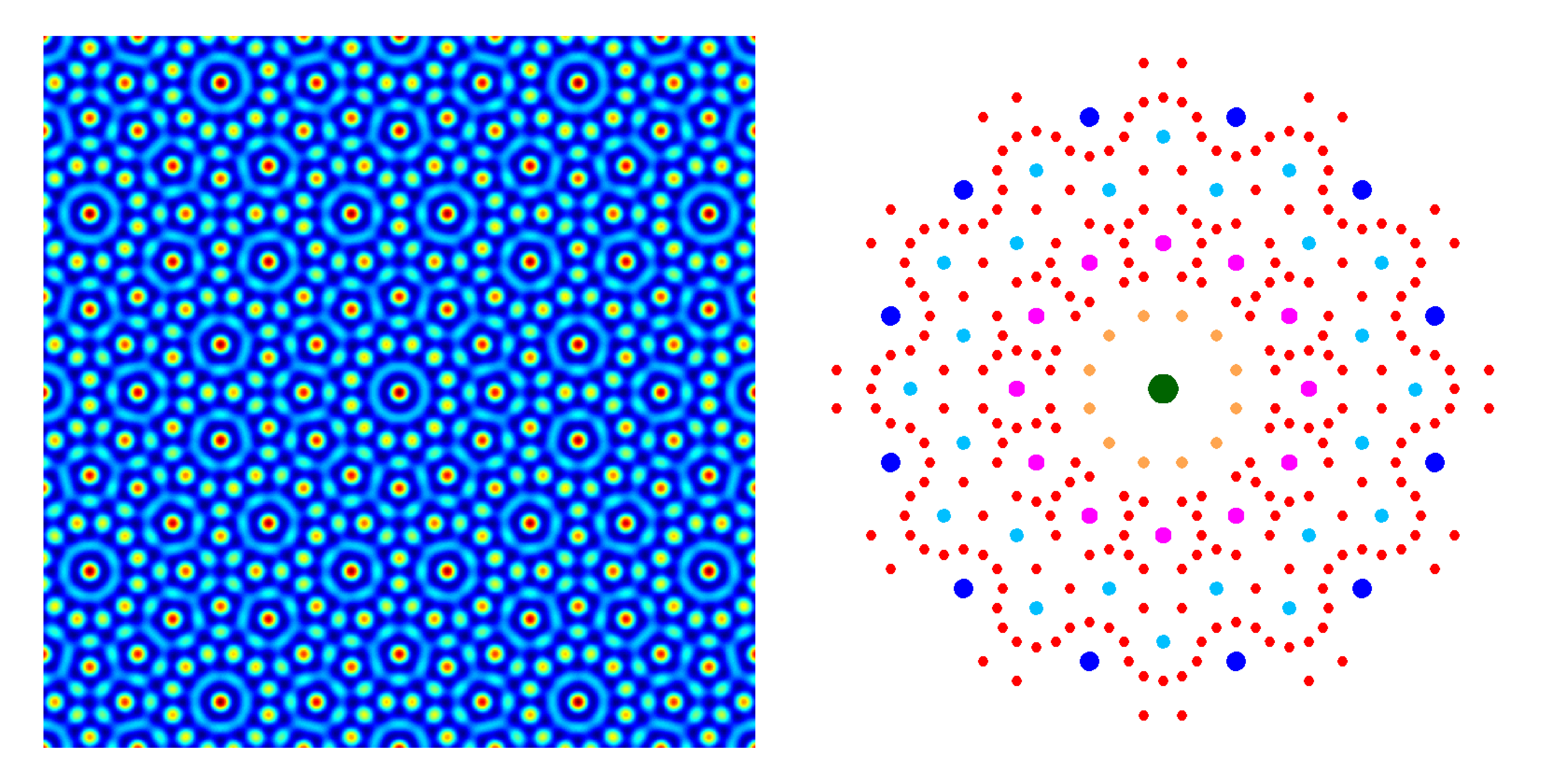}
	\caption{The stationary dodecagonal quasicrystal phase in LP model with 
	$ c = 24,\varepsilon = -6,\kappa = 6 $. \textbf{Left:} physical morphology; 
	\textbf{Right:} Fourier spectral points whose coefficient intensity is larger than 0.001}
	\label{fig:DDQCPhase}
\end{figure}

The parameters in LP models are set as $ c = 24,\varepsilon = -6,\kappa = 6 $, 
and $38^4$ wavefunctions are used to discretize LP energy functional. The convergent
stationary quasicrystal is given in \cref{fig:DDQCPhase},
including its order parameter distribution and Fourier spectrum. 
The numerical behavior of different approaches can be found in \cref{fig:DDQC:comparison}.
To better observe the change tendency, we use the convergent energy value $ E_{s} =
-15.97486323815640 $ as a baseline to show the relative energy changes against with iterations. 
We find again that our proposed approaches are more efficient than others.

\begin{figure}[htbp]
	\centering
	\includegraphics[width=5.0in]{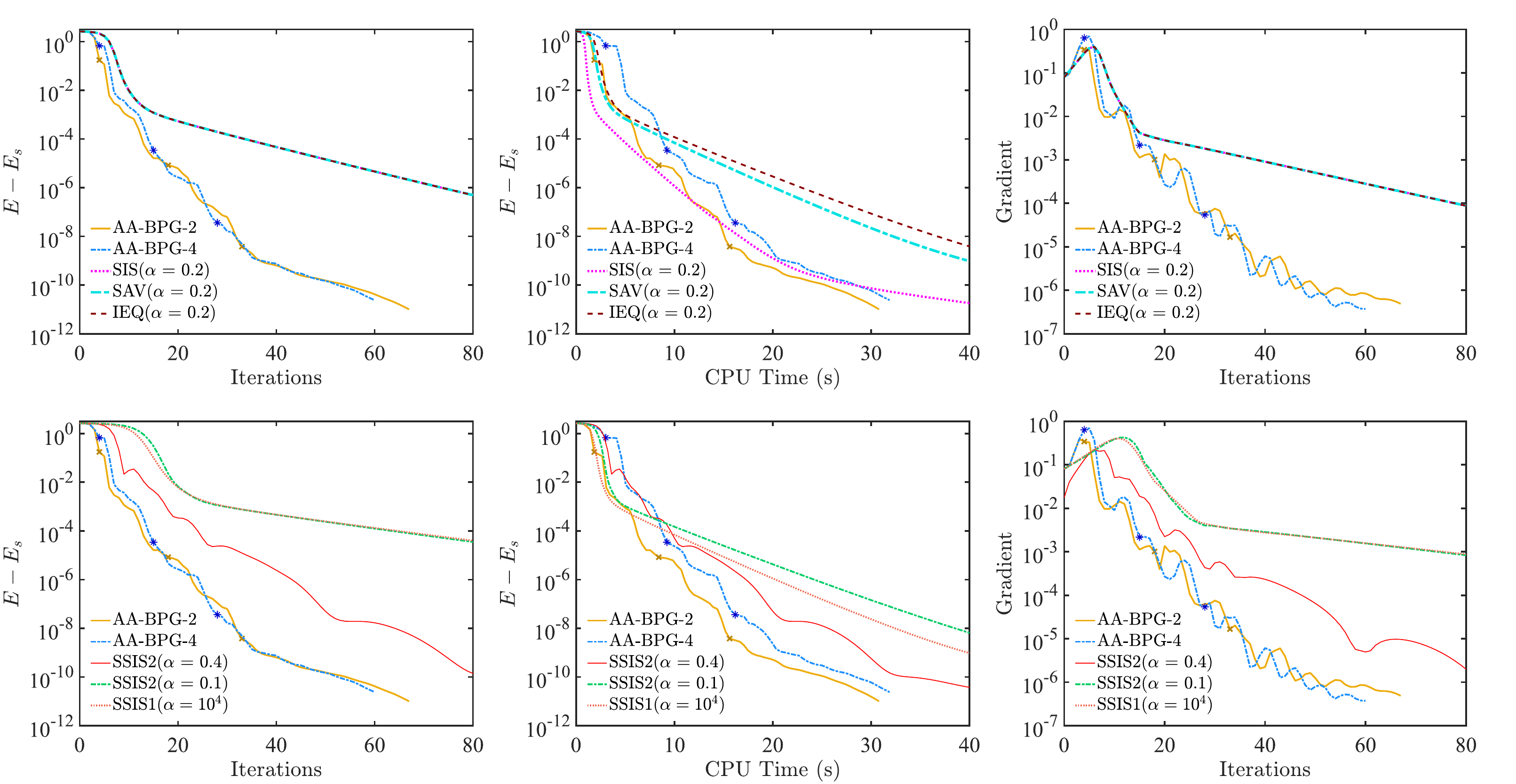}
	\caption{
	Dodecagonal quasicrystal: comparisons of numerical behaviors of the 
	AA-BPG-2/4 approaches with other numerical methods.
	The details of these images are the same with \cref{fig:DG:comparison}.
	}
	\label{fig:DDQC:comparison}
\end{figure}

\subsection{Local acceleration}
\label{subsec:acceleration}
{
The motivation of the hybrid method is providing a framework to
locally accelerate the existing methods.
Certainly, the Newton-PCG method is suitable for all alternative methods mentioned above. 
In the \cref{fig:rate}, we give a detailed comparison of our Newton-PCG method applied to alternative methods. 
For method M, the acceleration ratio is defined as 
\begin{equation}
    \text{Acceleration ratio} := \frac{\text{CPU times of original method M} }{\text{CPU times of hybrid method N-M}}.
\end{equation}
All numerical parameters, such as step size, of all alternative approaches are 
keep the same as former to guarantee the best performance. 
To launch the Newton-PCG method, we choose the gradient difference {$\|g_k - g_{k-1}\| <
10^{-3}$  } in computing crystal and energy difference $|E(\hPhi^k) - E(\hPhi^{k-1})| < 10^{-4}$ 
in computing the quasicrystal as the measurement. As shown in our numerical results, 
our Hessian based methods can accelerate all the existing methods with acceleration ratio 
ranging from 2-14. After using the proposed local acceleration, we observe that 
all the compared approaches have similar performance in terms of the CPU time. Moreover, 
it is noted that the acceleration ratio for the AA-BPG-2 method is the smallest one as 
is shows the best performance without coupling the Newton-PCG method.
}
\begin{figure}[htbp]
 	\centering
 	\includegraphics[scale=0.55]{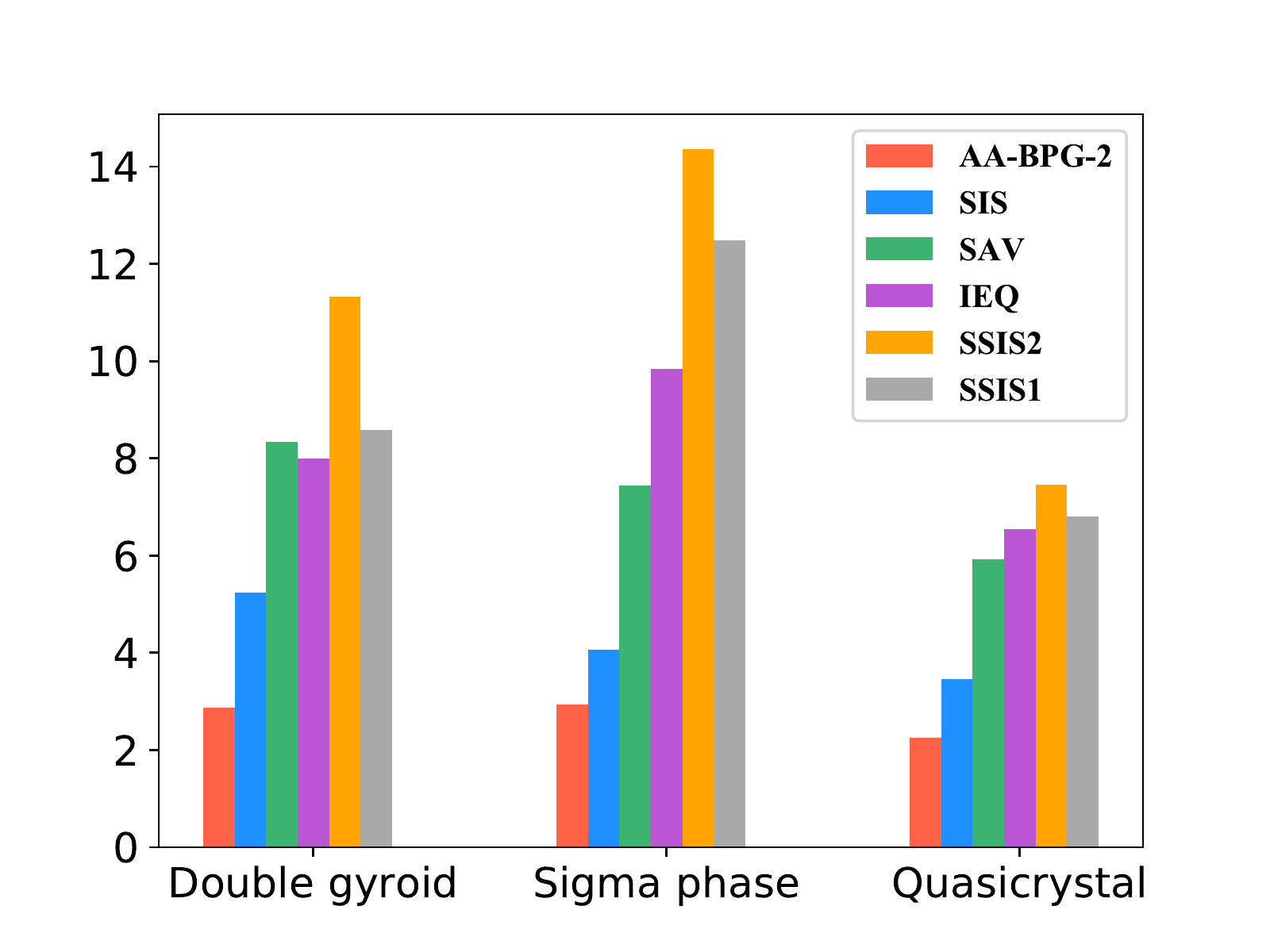}
 	\caption{The acceleration ratio of applying Newton-PCG algorithm to existing methods compared 
 	with original ones for computing periodic crystals and quasicrystals}
 	\label{fig:rate}
\end{figure}

\section{Conclusion}\label{sec:conclusion}

In this paper, efficient and robust computational approaches have been proposed to 
find the stationary states of PFC models. Instead of formulating the energy minimization 
as a gradient flow, we applied the modern optimization methods directly on the discretized 
energy with mass conservation and energy dissipation. Moreover, the AA-BPG methods 
with suitable choice of $ h $ overcome the global Lipschitz constant requirement 
in theoretical analysis and the step sizes are adaptively obtained by line search technique. 
We also propose a practical Newton-PCG method and introduce a hybrid framework to further 
accelerate the local convergence of gradient based methods. Extensive results in computing 
periodic crystals and quasicrystals show their advantages in terms of computation efficiency. 
{Thus, it motivates us to continue finding the deep relationship between the gradient flow and the 
optimization, applying our methods to many related problems, such as SPFC, MPFC models, 
and extending to more spatial discretization methods.}

%


\input{references}
\section*{Appendix A: Proof of \cref{Thm:BAPGconvergence}}
	Before prove the convergent property, we first present a useful lemma for our analysis.
	\begin{lemma}[Uniformized Kurdyka-Lojasiewicz property
		\cite{bolte2014proximal}] Let $\Omega$ be a compact set and
		$E$ is constant on $\Omega$. Then, there exist $\epsilon>0$, $\eta>0$, and $\psi\in\Psi_\eta$ 
		such that for all $\bar u\in\Omega$ and all $u\in\Gamma_\eta(\bar u,\epsilon)$, 
		one has,
		\begin{equation}\label{UKL}
		\psi^{'}(E(u)-E(\bar u))\dist(\vzero,\partial E(u))\geq 1,
		\end{equation}
		where $\Psi_\eta =\{\psi\in C[0,\eta)\cap C^1(0,\eta), \psi \text{ is concave}, \psi(0)=0, \psi^{'}>0 \text{ on } (0,\eta)\}$ and $\Gamma_\eta(x,\epsilon) = \{y|\|x-y\|\leq \epsilon, E(x)<E(y)<E(x)+\eta\}$.
		\label{lemma:ukl}
	\end{lemma}

	Now, we show the proof of \cref{Thm:BAPGconvergence}, which is similar to the framework in\,\cite{Heinz}.
	\begin{proof}
	Let $S(x^0)$ be the set
	of limiting points of the sequence $\{x^k\}_{k=0}^{\infty}$ starting from
	$x^0$. By the boundedness of $\{x^k\}_{k=0}^{\infty}$ and the fact
	$S(x^0)=\cap_{q\in\mathbb{N}}\overline{\cup_{k\geq
			q}\{x^k\}}$, it follows that $S(x^0)$ is a
	non-empty and compact set. Moreover, from \cref{ABPGsuffDec}, we
	know $E(x)$ is constant on $S(x^0)$, denoted by
	$E^*$. If there exists some $k_0$ such that
	$E(x^{k_0})=E^*$, then we have $E(x^k)=E^*$ for all
	$k\geq k_0$ which is from \cref{ABPGsuffDec}. In the following
	proof, we assume that $E(x^k)>E^*$ for all $k$. Therefore,
	$\forall \epsilon,\eta>0$, there exists some $\ell>0$ such
	that for all $k>\ell$, we have
	$\mathrm{dist}(S(x^0),x^k)\leq \epsilon \text{ and
	} E^*<E(x^k)<E^*+\eta$, i.e. 
	\begin{equation}\label{APP:EQ}
	x\in\Gamma_{\eta}(x^*,\epsilon) \quad \text{for all }\quad x^*\in S(x^0).
	\end{equation}
	Applying \cref{lemma:ukl} for all $k>\ell$ we have
	\begin{equation*}
	\psi^{'}(E(x^k)-E^*)\dist(\vzero, E(x^k))\geq 1.
	\end{equation*}
	Form \cref{gradbound}, it implies
	\begin{equation}\label{KL1}
	\psi^{'}(E(x^k)-E^*)\geq
	\frac{1}{c_1(\|x^k-x^{k-1}\|+\bar{w}\|x^{k-1}-x^{k-2}\|)}.
	\end{equation}
	By the convexity of $\psi$, we have
	\begin{equation}\label{concave}
	\psi(E(x^k)-E^*) - \psi(E(x^{k+1}) -E^*)
	\geq \psi^{'}(E(x^k)-E^*)(E(x^k)-E(x^{k+1})).
	\end{equation}
	Define $\Delta_{p,q}=\psi(E(x^p)-E^*) - \psi(E(x^q) -E^*)$ and $C=(1+\bar{w})c_1/c_0>0$.
	Together with \cref{KL1}, \cref{concave} and \cref{ABPGsuffDec}, we have for all $k>\ell$
	\begin{equation}
	\Delta_{k,k+1}\geq \frac{c_0\|x^{k+1}-x^k\|^2}{c_1(\|x^k-x^{k-1}\|+\bar{w}\|x^{k-1}-x^{k-2}\|)}\geq \frac{\|x^{k+1}-x^k\|^2}{C(\|x^k-x^{k-1}\|+\|x^{k-1}-x^{k-2}\|)}.
	\end{equation}
	Therefore,
	\begin{equation}\label{GeoIneq}
	2\|x^{k+1}-x^{k}\|\leq \dfrac{1}{2}(\|x^k-x^{k-1}\|+\|x^{k-1}-x^{k-2}\|) + 2C\Delta_{k,k+1},
	\end{equation}
	which is from the geometric inequality. For any $k>\ell$, summing up
	\cref{GeoIneq} for $i=\ell+1,\ldots,k$, it implies
	\begin{equation*}
	\begin{split}
	& 2\sum_{i=\ell+1}^k \|x^{i+1}-x^{i}\|\leq {\dfrac{1}{2}}\sum_{i=\ell+1}^k (\|x^i-x^{i-1}\|+\|x^{i-1}-x^{i-2}\|)+{2C}\sum_{i=\ell+1}^k\Delta_{i,i+1}\\
	\leq & \sum_{i=\ell+1}^k \|x^{i+1}-x^i\| + \|x^{\ell+1}-x^\ell\|+ \|x^\ell-x^{\ell-1}\|+2C\Delta_{\ell+1,k+1},
	\end{split}
	\end{equation*}
	where the last inequality is from the fact that $\Delta_{p,q}+\Delta_{q,r}=\Delta_{p,r}$ for all $p,q,r\in\mathbb{N}$. Since $\psi\geq 0$, for any $k>\ell$ and
	we have
	\begin{equation}\label{boundGlobalConv}
	\begin{split}
	\sum_{i=\ell+1}^k\| x^{i+1} -x^i\| \leq
	\|x^{\ell+1}-x^\ell\|+ \|x^\ell-x^{\ell-1}\|+2C\psi(E(x^{\ell+1})-E^*).
	\end{split}
	\end{equation}
	This easily implies that
	$\sum_{k=1}^{\infty}\|x^{k+1}-x^k\|<\infty$. Together with \cref{Thm:BAPGconvergence}, we obtain
	\begin{equation*}
	\lim\limits_{k\rightarrow+\infty}x^k=x^*,\quad \vzero \in\partial E(x^{*})= 0.
	\end{equation*}
\end{proof}

\section*{Appendix B: Proof of \cref{Them:innerprod}}
	The proof is similar to the framework in\,\cite{zhao2010newton}. Let $ x^* $ be the exact solution 
	and $ e_i = x^*- x^i $ for all $ i $. We first prove some important properties of \cref{alg:PCG}.
	
	\textbf{Property I: $ r_i = Ax^i - b $.}
	From the step 4 of  \cref{alg:PCG}, we have $ \alpha_iAp_{i-1} = Ax^{i} - Ax^{i-1} . $ Then,
	\begin{align*}
	r_{i} &= r_{i-1} + \alpha_i Ap_{i-1}  = r_0 + \sum_{j=1}^i \alpha_jAp_{j-1} = -b + \sum_{j=1}^i \alpha_jAp_{j-1} \\
	&= -b + \sum_{j=1}^i(Ax^j - Ax^{j-1})  = -b + Ax^{i} - Ax^0 = Ax^i  - b.			
	\end{align*}
	
	\textbf{Property II: $ \langle p_i, b\rangle  = \|r_i\|_{M^{-1}}^2~(i = 0,1,2,\cdots)$.}
	By the formula (5.40) in \,\cite{NumericalOpt}, we know that
	$ 	\la r_i,r_j\ra_{M^{-1}} = 0 ~(i\neq j)$. Together with the definition of $ \beta_i $ and $ p_i $ in \cref{alg:PCG}, we get 
	\begin{align}
	\begin{split}
	\la p_0, b\ra &= \la p_0,-r_0 \ra =\la M^{-1}r_0, r_0\ra =  \|r_0\|_{M^{-1}}^2,\\
	\la p_i, b\ra &= \la p_i, -r_0\ra = \la M^{-1}r_i, r_0\ra + \beta_i\la p_{i-1}, -r_0\ra = \beta_i\la p_{i-1}, -r_0\ra = \left(\prod_{j=1}^i\beta_i\right)\la p_{0}, -r_0\ra \\&=\left(\prod_{j=1}^i\beta_i\right) \| r_0\|_{M^{-1}}^2
	= \left(\prod_{j=2}^i\beta_i\right)\|r_1\|_{M^{-1}}^2 =  \|r_i\|_{M^{-1}}^2,\quad \forall i=1,2,\cdots,
	\end{split}	
	\label{eq:4.9}
	\end{align}
	
	\textbf{Property III: $ \|e_i\|_A \geq \|e_{i+1}\|_A$.} According to the iteration of $ p_i $, on has
	\begin{align}
	\begin{split}
	\la p_i, -r_{i+1} \ra&= \la -M^{-1}r_i + \beta_ip_{i-1}, -r_{i+1}\ra =0+ \beta_i\la p_{i-1}, -r_{i+1} \ra 
	\\&= \left(\prod_{j=1}^i\beta_j\right)\la p_0, -r_{i+1} \ra = \left(\prod_{j=1}^i\beta_j\right)\la M^{-1}r_0, r_{i+1} \ra = 0.
	\end{split}
	\label{eq:zeroprod}
	\end{align}
	By the property I, we have $  Ae_{i+1} = A(x^* -x^{i+1})= b - Ax^{i+1}  = -r_{i+1} $, which implies $ \la p_i, Ae_{i+1} \ra = 0 $.
	Using the fact that $ e_i = e_{i+1} + x^{i+1} - x^{i} = e_{i+1} +\alpha_{i+1}p_i $, the following equation holds for all $ i\geq 0 $:
	\begin{align}
	\begin{split}
	\|e_i\|_{A}^2 &= \|e_{i+1} +\alpha_{i+1}p_i\|_A^2 =
	\|e_{i+1}\|_{A}^2 + 2\alpha_{i+1}\la p_i, Ae_{i+1} \ra + \|\alpha_{i+1}p_i\|_{A}^2\\
	&=\|e_{i+1}\|_{A}^2+ \alpha_{i+1}^2\|p_i\|_{A}^2 \geq \|e_{i+1}\|_{A}^2.
	\end{split}	
	\label{eq:4.10}	
	\end{align}
	
	\textbf{Property IV: $ \la x^i, b\ra \geq \la x^{i-1}, b\ra.$}
	The definition of $ \alpha_{j} $ gives	$ \|r_{j-1}\|_{M^{-1}}^2 = \alpha_j\|p_{j-1}\|_{A}^2 $. Together with \cref{eq:4.9} and  \cref{eq:4.10}, we have
	\begin{align}
	\begin{split}
	\la x^i, b\ra& = \la x^{i-1}, b\ra + \la \alpha_ip_{i-1}, b\ra = \la x^0, b\ra + \sum_{j=1}^i\la \alpha_jp_{j-1}, b\ra = \sum_{j=1}^i \alpha_j\|r_{j-1}\|_{M^{-1}}^2\\ &= \sum_{j=1}^i \alpha_j^2\|p_{j-1}\|_{A}^2 = \sum_{j=1}^i(\|e_{j-1}\|_{A}^2 - \|e_{j}\|_{A}^2)= \|e_0\|_{A}^2 - \|e_i\|_{A}^2,
	\end{split}
	\label{eq:4.11}
	\end{align}
	which implies  $ \la x^i, b\ra \geq \la x^{i-1}, b\ra$ by the monotonicity of $ \|e_i\|_A^2 $.
	
	Now, we can prove the main result. By using the definition of $ p_0 $ and $ \alpha_1 $, we obtain
	\begin{align}
	\begin{split}
	\dfrac{\la x^i,b\ra}{\|b\|^2} &\geq \dfrac{\la x^1,b\ra}{\|b\|^2} = \dfrac{\la x^0 + \alpha_1 p_0,b\ra}{\|b\|^2} =  \alpha_1\dfrac{\la  p_0,b\ra}{\|b\|^2} = \dfrac{\la r_0, p_0\ra}{\la p_0, A p_0 \ra}\dfrac{\la  M^{-1}b,b\ra}{\|b\|^2} \\
	&=  \dfrac{\la Mp_0, p_0\ra}{\la p_0, A p_0 \ra}\dfrac{\la  M^{-1}b,b\ra}{\|b\|^2} \geq \dfrac{\la Mp_0, p_0\ra}{\la p_0, A p_0 \ra}\dfrac{1}{\lambda_{\max}(M)}.
	\end{split}
	\label{4.18}
	\end{align}
	Since $ M $ is positive,  we know $ M = M^{1/2}M^{1/2} $, where $ M^{1/2} $ is still positive. As a result, 
	\begin{align}
	\|M\| = \lambda_{max}(M)  =\lambda_{\max}(M^{1/2}M^{1/2}) = \lambda^2_{\max}(M^{1/2}) = \|M^{1/2}\|^2.
	\label{eq:6.8}
	\end{align}
	Let $ y = M^{1/2}p_0 $, we get
	\begin{align}
	\begin{split}
	\dfrac{\la Mp_0, p_0\ra}{\la p_0, A p_0 \ra} &= \dfrac{\la y,y\ra}{\la y, M^{-1/2}AM^{-1/2}y\ra} \geq \dfrac{1}{\lambda_{\max}(M^{-1/2}AM^{-1/2})} = \dfrac{1}{\|M^{-1/2}AM^{-1/2}\|}\\
	&\geq \dfrac{1}{\|M^{-1/2}\|\cdot \|A\|\cdot \|M^{-1/2}\|} = \dfrac{\|M\|}{ \|A\|} = \dfrac{\lambda_{\max}(M)}{\lambda_{\max}(A)}.
	\end{split}
	\label{4.19}
	\end{align}
	where the second inequality takes the fact that $ \|AB\| \leq \|A\| \cdot \|B\| $. Together with  \cref{4.18}, we get
	\begin{align}
	\dfrac{\la x^i,b\ra}{\|b\|^2}   \geq \dfrac{\la Mp_0, p_0\ra}{\la p_0, A p_0 \ra}\dfrac{1}{\lambda_{\max}(M)} \geq \dfrac{\lambda_{\max}(M)}{\lambda_{\max}(A)}\dfrac{1}{\lambda_{\max}(M)} = \dfrac{1}{\lambda_{\max}(A)}.
	\end{align}
	To verify another inequality, we use \cref{eq:4.11} and the fact that $ e_0 = x^* - x^0 = -A^{-1}b $, 
	\begin{align*}
	\dfrac{\la x^i,b\ra}{\|b\|^2}\ = \dfrac{\|e_0\|_A^2 - \|e_i\|_A^2}{\|b\|^2} \leq \dfrac{\|e_0\|_{A}^2}{\|b\|^2} = \dfrac{\|A^{-1}b \|_{A}^2}{\|b\|^2} = \dfrac{\la b, A^{-1}b\ra}{\|b\|^2} \leq \dfrac{1}{\lambda_{\min}(A)}.
	\end{align*}

\end{document}

%% file: ex_shared.tex
\usepackage{mathrsfs,amsmath,amssymb,bm}
\usepackage{lipsum}
\usepackage{amsfonts}
\usepackage{graphicx}
\usepackage{epstopdf}
\usepackage{makecell, rotating}
\usepackage{multirow}
\usepackage{algorithmic}

\newcommand{\bbR}{\mathbb{R}}
\newcommand{\bbC}{\mathbb{C}}
\newcommand{\bbZ}{\mathbb{Z}}

\newcommand{\bbN}{\mathbb{N}}

\newcommand{\br}{\bm{r}}
\newcommand{\bh}{\bm{h}}

\newcommand{\bN}{\bm{N}}
\newcommand{\bk}{\bm{k}}

\newcommand{\bB}{\mathbf{B}}

\newcommand{\bI}{\mathbf{I}}
\newcommand{\calB}{\mathcal{B}}
\newcommand{\calS}{\mathcal{S}}
\newcommand{\calP}{\mathcal{P}}

\newcommand{\calM}{\mathcal{M}}
\newcommand{\calJ}{\mathcal{J}}
\newcommand{\calF}{\mathcal{F}}
\newcommand{\calI}{\mathcal{I}}

\newcommand{\hphi}{\hat{\phi}}
\newcommand{\hPhi}{\hat{\Phi}}

\newcommand{\vzero}{\mathbf{0}}
\newcommand{\tE}{\tilde{E}}
\newcommand{\tF}{\tilde{F}}
\newcommand{\tG}{\tilde{G}}
\newcommand{\tg}{\tilde{g}}

\newcommand{\tJ}{\tilde{\mathcal{J}}}

\newcommand{\la}{\langle}
\newcommand{\ra}{\rangle}

\newcommand{\dom}{\mathrm{dom}}
\newcommand{\dist}{\mathrm{dist}}
\DeclareMathOperator{\intdom}{\mathrm{int}\dom}
\newcommand{\cS}{\mathcal{S}}
\newcommand{\bzero}{\mathbf{0}}
\newcommand{\hPsi}{\hat{\Psi}}

\newtheorem{thm}{Theorem}[section]

\newtheorem{remark}[thm]{Remark}

\newtheorem{assumption}[thm]{Assumption}

\newcommand\tbbint{{-\mkern -16mu\int}}

\newcommand\dbbint{{-\mkern -19mu\int}}

\newcommand\bbint{
	{\mathchoice{\dbbint}{\tbbint}{\tbbint}{\tbbint}}
}

\DeclareMathOperator*{\argmin}{\mathrm{argmin}}

\usepackage[draft]{changes}
\usepackage{lipsum}
\definechangesauthor[name={Kai Jiang}, color=red]{JK}
\definechangesauthor[name={Chenglong Bao}, color=blue]{bao}
\definechangesauthor[name={Chang Chen}, color=violet]{CC}

\newsiamremark{hypothesis}{Hypothesis}
\crefname{hypothesis}{Hypothesis}{Hypotheses}
\newsiamthm{claim}{Claim}

\headers{Efficient numerical methods for PFC
models}{K. Jiang, W. Si, C. Chen and C. Bao.}

\title{Efficient numerical methods for computing the stationary states of phase
field crystal models
}

\author{Kai Jiang\thanks{School of Mathematics and Computational Science, 
 Hunan Key Laboratory for Computation and Simulation in Science and Engineering,
Xiangtan University, Xiangtan, Hunan, China, 411105.} \and Wei Si\footnotemark[1] \and Chang Chen\footnotemark[1] \and Chenglong Bao\thanks{Corresponding author. Yau Mathematical Sciences Center, Tsinghua University, Beijing, China, 100084.
		(\email{clbao@mail.tsinghua.edu.cn}).}
}

\usepackage{amsopn}

\makeatletter
\newcommand*{\addFileDependency}[1]{
  \typeout{(#1)}
  \@addtofilelist{#1}
  \IfFileExists{#1}{}{\typeout{No file #1.}}
}
\makeatother


%% file: NewtonAPG.bbl
\begin{thebibliography}{10}

\bibitem{barzilai1988two}
{\sc J.~Barzilai and J.~M. Borwein}, {\em Two-point step size gradient
  methods}, IMA journal of numerical analysis, 8 (1988), pp.~141--148.

\bibitem{Heinz}
{\sc H.~H. Bauschke, J.~Bolte, and M.~Teboulle}, {\em A descent lemma beyond
  lipschitz gradient continuity: first-order methods revisited and
  applications}, Mathematics of Operations Research, 42 (2017), pp.~330--348.

\bibitem{beck2009fast}
{\sc A.~Beck and M.~Teboulle}, {\em A fast iterative shrinkage-thresholding
  algorithm for linear inverse problems}, SIAM journal on imaging sciences, 2
  (2009), pp.~183--202.

\bibitem{bolte2014proximal}
{\sc J.~Bolte, S.~Sabach, and M.~Teboulle}, {\em Proximal alternating
  linearized minimization for nonconvex and nonsmooth problems}, Mathematical
  Programming, 146 (2014), pp.~459--494.

\bibitem{bolte2018first}
{\sc J.~Bolte, S.~Sabach, M.~Teboulle, and Y.~Vaisbourd}, {\em First order
  methods beyond convexity and lipschitz gradient continuity with applications
  to quadratic inverse problems}, SIAM Journal on Optimization, 28 (2018),
  pp.~2131--2151.

\bibitem{brazovskii1975phase}
{\sc S.~Brazovskii}, {\em Phase transition of an isotropic system to a
  nonuniform state}, Soviet Journal of Experimental and Theoretical Physics, 41
  (1975), p.~85.

\bibitem{bregman1967relaxation}
{\sc L.~M. Bregman}, {\em The relaxation method of finding the common point of
  convex sets and its application to the solution of problems in convex
  programming}, USSR computational mathematics and mathematical physics, 7
  (1967), pp.~200--217.

\bibitem{brezis2010functional}
{\sc H.~Brezis}, {\em Functional analysis, Sobolev spaces and partial
  differential equations}, Springer Science \&amp; Business Media, 2010.

\bibitem{chen2002phase}
{\sc L.~Chen}, {\em Phase-field models for microstructure evolution}, Annual
  review of materials research, 32 (2002), pp.~113--140.

\bibitem{chen1998applications}
{\sc L.~Q. Chen and J.~Shen}, {\em Applications of semi-implicit
  fourier-spectral method to phase field equations}, Computer Physics
  Communications, 108 (1998), pp.~147--158.

\bibitem{cheng2019energy}
{\sc K.~Cheng, C.~Wang, and S.~M.~Wise}, {\em An energy stable bdf2 fourier
  pseudo-spectral numerical scheme for the square phase field crystal
  equation}, Communications in Computational Physics, 26 (2019),
  pp.~1335--1364.

\bibitem{du2008adaptive}
{\sc Q.~Du and J.~Zhang}, {\em Adaptive finite element method for a phase field
  bending elasticity model of vesicle membrane deformations}, SIAM Journal on
  Scientific Computing, 30 (2008), pp.~1634--1657.

\bibitem{du2004stability}
{\sc Q.~Du and W.-x. Zhu}, {\em Stability analysis and application of the
  exponential time differencing schemes}, Journal of Computational Mathematics,
   (2004), pp.~200--209.

\bibitem{feng2007analysis}
{\sc X.~Feng, Y.~He, and C.~Liu}, {\em Analysis of finite element
  approximations of a phase field model for two-phase fluids}, Mathematics of
  computation, 76 (2007), pp.~539--571.

\bibitem{guo2018high}
{\sc R.~Guo and Y.~Xu}, {\em A high order adaptive time-stepping strategy and
  local discontinuous galerkin method for the modified phase field crystal
  equation}, Comput. Phys, 24 (2018), pp.~123--151.

\bibitem{hiller1985crystallographic}
{\sc H.~Hiller}, {\em The crystallographic restriction in higher dimensions},
  Acta Crystallographica Section A: Foundations of Crystallography, 41 (1985),
  pp.~541--544.

\bibitem{hu2009stable}
{\sc Z.~Hu, S.~M. Wise, C.~Wang, and J.~S. Lowengrub}, {\em Stable and
  efficient finite-difference nonlinear-multigrid schemes for the phase field
  crystal equation}, Journal of Computational Physics, 228 (2009),
  pp.~5323--5339.

\bibitem{jiang2015stability}
{\sc K.~Jiang, J.~Tong, P.~Zhang, and A.-C. Shi}, {\em Stability of
  two-dimensional soft quasicrystals in systems with two length scales},
  Physical Review E, 92 (2015), p.~042159.

\bibitem{jiang2013discovery}
{\sc K.~Jiang, C.~Wang, Y.~Huang, and P.~Zhang}, {\em Discovery of new
  metastable patterns in diblock copolymers}, Communications in Computational
  Physics, 14 (2013), pp.~443--460.

\bibitem{jiang2014numerical}
{\sc K.~Jiang and P.~Zhang}, {\em Numerical methods for quasicrystals}, Journal
  of Computational Physics, 256 (2014), pp.~428--440.

\bibitem{katznelson2004anintroduction}
{\sc Y.~Katznelson}, {\em An introduction to harmonic analysis}, 2004.

\bibitem{lee2017first}
{\sc H.~G. Lee, J.~Shin, and J.-Y. Lee}, {\em First-and second-order energy
  stable methods for the modified phase field crystal equation}, Computer
  Methods in Applied Mechanics and Engineering, 321 (2017), pp.~1--17.

\bibitem{lee2010discovery}
{\sc S.~Lee, M.~Bluemle, and F.~Bates}, {\em Discovery of a frank-kasper
  $\sigma$ phase in sphere-forming block copolymer melts}, Science, 330 (2010),
  p.~349.

\bibitem{Bregman}
{\sc Q.~Li, Z.~Zhu, G.~Tang, and M.~B. Wakin}, {\em Provable bregman-divergence
  based methods for nonconvex and non-lipschitz problems}, arXiv preprint
  arXiv:1904.09712,  (2019).

\bibitem{lifshitz2007soft}
{\sc R.~Lifshitz and H.~Diamant}, {\em Soft quasicrystals--why are they
  stable?}, Philosophical Magazine, 87 (2007), pp.~3021--3030.

\bibitem{lifshitz1997theoretical}
{\sc R.~Lifshitz and D.~Petrich}, {\em Theoretical model for faraday waves with
  multiple-frequency forcing}, Physical review letters, 79 (1997),
  pp.~1261--1264.

\bibitem{liu2015analysis}
{\sc X.~Liu, Z.~Wen, X.~Wang, M.~Ulbrich, and Y.~Yuan}, {\em On the analysis of
  the discretized kohn--sham density functional theory}, SIAM Journal on
  Numerical Analysis, 53 (2015), pp.~1758--1785.

\bibitem{mkhonta2013exploring}
{\sc S.~K. Mkhonta, K.~R. Elder, and Z.-F. Huang}, {\em Exploring the complex
  world of two-dimensional ordering with three modes}, Phys. Rev. Lett., 111
  (2013), p.~035501, \url{https://doi.org/10.1103/PhysRevLett.111.035501},
  \url{https://link.aps.org/doi/10.1103/PhysRevLett.111.035501}.

\bibitem{NumericalOpt}
{\sc J.~Necedal and S.~J.Wright}, {\em Numerical Optimization}, Springer, 2006.

\bibitem{o2015adaptive}
{\sc B.~O'donoghue and E.~Candes}, {\em Adaptive restart for accelerated
  gradient schemes}, Foundations of computational mathematics, 15 (2015),
  pp.~715--732.

\bibitem{provatas2010phase}
{\sc N.~Provatas and K.~Elder}, {\em Phase-field methods in materials science
  and engineering}, Wiley-VCH, 2010.

\bibitem{shen2019new}
{\sc J.~Shen, J.~Xu, and J.~Yang}, {\em A new class of efficient and robust
  energy stable schemes for gradient flows}, SIAM Review, 61 (2019),
  pp.~474--506.

\bibitem{shen2010numerical}
{\sc J.~Shen and X.~Yang}, {\em Numerical approximations of allen-cahn and
  cahn-hilliard equations}, Discrete Contin. Dyn. Syst, 28 (2010),
  pp.~1669--1691.

\bibitem{shi1996theory}
{\sc A.-C. Shi, J.~Noolandi, and R.~C. Desai}, {\em Theory of anisotropic
  fluctuations in ordered block copolymer phases}, Macromolecules, 29 (1996),
  pp.~6487--6504.

\bibitem{shin2016first}
{\sc J.~Shin, H.~G. Lee, and J.-Y. Lee}, {\em First and second order numerical
  methods based on a new convex splitting for phase-field crystal equation},
  Journal of Computational Physics, 327 (2016), pp.~519--542.

\bibitem{swift1977hydrodynamic}
{\sc J.~Swift and P.~C. Hohenberg}, {\em Hydrodynamic fluctuations at the
  convective instability}, Phys. Rev. A, 15 (1977), pp.~319--328,
  \url{https://doi.org/10.1103/PhysRevA.15.319},
  \url{http://link.aps.org/doi/10.1103/PhysRevA.15.319}.

\bibitem{tseng2008accelerated}
{\sc P.~Tseng}, {\em On accelerated proximal gradient methods for
  convex-concave optimization}, submitted to SIAM Journal on Optimization, 2
  (2008), p.~3.

\bibitem{ulbrich2015proximal}
{\sc M.~Ulbrich, Z.~Wen, C.~Yang, D.~Klockner, and Z.~Lu}, {\em A proximal
  gradient method for ensemble density functional theory}, SIAM Journal on
  Scientific Computing, 37 (2015), pp.~A1975--A2002.

\bibitem{vanconvex}
{\sc J.~Van~Tiel}, {\em Convex analysis: An introductory text, 1984}.

\bibitem{wang2011energy}
{\sc C.~Wang and S.~M. Wise}, {\em An energy stable and convergent
  finite-difference scheme for the modified phase field crystal equation}, SIAM
  Journal on Numerical Analysis, 49 (2011), pp.~945--969.

\bibitem{wise2009energy}
{\sc S.~M. Wise, C.~Wang, and J.~S. Lowengrub}, {\em An energy-stable and
  convergent finite-difference scheme for the phase field crystal equation},
  SIAM Journal on Numerical Analysis, 47 (2009), p.~2269.

\bibitem{wu2017regularized}
{\sc X.~Wu, Z.~Wen, and W.~Bao}, {\em A regularized newton method for computing
  ground states of bose--einstein condensates}, Journal of Scientific
  Computing, 73 (2017), pp.~303--329.

\bibitem{xiao2018regularized}
{\sc X.~Xiao, Y.~Li, Z.~Wen, and L.~Zhang}, {\em A regularized semi-smooth
  newton method with projection steps for composite convex programs}, Journal
  of Scientific Computing, 76 (2018), pp.~364--389.

\bibitem{xie2014sigma}
{\sc N.~Xie, W.~Li, F.~Qiu, and A.-C. Shi}, {\em $\sigma$ phase formed in
  conformationally asymmetric ab-type block copolymers}, Acs Macro Letters, 3
  (2014), pp.~906--910.

\bibitem{xu2006stability}
{\sc C.~Xu and T.~Tang}, {\em Stability analysis of large time-stepping methods
  for epitaxial growth models}, SIAM Journal on Numerical Analysis, 44 (2006),
  pp.~1759--1779.

\bibitem{yang2016linear}
{\sc X.~Yang}, {\em Linear, first and second-order, unconditionally energy
  stable numerical schemes for the phase field model of homopolymer blends},
  Journal of Computational Physics, 327 (2016), pp.~294--316.

\bibitem{zhao2010newton}
{\sc X.-Y. Zhao, D.~Sun, and K.-C. Toh}, {\em A newton-cg augmented lagrangian
  method for semidefinite programming}, SIAM Journal on Optimization, 20
  (2010), pp.~1737--1765.

\end{thebibliography}
